\documentclass[
a4paper,
11pt,
twoside,
]{article}
\usepackage[utf8]{inputenc}
\usepackage[T1]{fontenc}
\usepackage[sc]{mathpazo}
\usepackage[english]{babel}
\usepackage{geometry}
\usepackage{amsmath,amsfonts,amssymb,amsthm}
\usepackage[final]{hyperref}
\usepackage{graphicx}
\usepackage{mathtools}
\usepackage[dvipsnames]{xcolor}
\usepackage{pdflscape}
\usepackage{etoolbox}
\usepackage{fancyhdr}
\usepackage{lastpage}
\usepackage{ifdraft}
\usepackage{hyphenat}
\ifdraft{\usepackage{showlabels}}{}
\usepackage{enumitem}
\usepackage{siunitx}
\usepackage[mathlines,pagewise]{lineno}

\newcommand{\dd}{\mathop{}\!\mathrm{d}}

\theoremstyle{plain}
\newtheorem{theorem}{Theorem}[section]

\newtheorem{proposition}[theorem]{Proposition}

\newtheorem{remark}[theorem]{Remark}

\numberwithin{equation}{section}
\numberwithin{table}{section}
\numberwithin{figure}{section}

\fancypagestyle{mypagestyle}{%
  \fancyhf{}%
  \fancyhead[LO,RE]{\scriptsize Computing periodic solutions of REs}%
  \fancyhead[RO,LE]{\scriptsize A. And\`o and D. Breda}%
  \fancyfoot[LE,RO]{\scriptsize \thepage\,/\,\pageref*{LastPage}}%
}
\title{Piecewise orthogonal collocation for computing periodic solutions of renewal equations}
\author{
Alessia And\`o$^{1,2,3}$ and Dimitri Breda$^{2,4}$\\[.5em]
\small $^{1}$Area of Mathematics, Gran Sasso Science Institute\\[-.2em]
\small via Francesco Crispi 7, 67100 L'Aquila, Italy\\[.5em]
\small $^{2}$CDLab -- Computational Dynamics Laboratory\\[-.2em]
\small Department of Mathematics, Computer Science and Physics -- University of Udine\\[-.2em]
\small via delle scienze 206, 33100 Udine, Italy\\[.5em]
\small $^{3}$\texttt{alessia.ando@gssi.it}\\[.5em]
\small $^{4}$\texttt{dimitri.breda@uniud.it}
}
\date{\today}

\begin{document}
\clearpage
\maketitle
\thispagestyle{empty}

\begin{abstract}We extend the use of piecewise orthogonal collocation to computing periodic solutions of renewal equations, which are particularly important in modeling population dynamics. We prove convergence through a rigorous error analysis. Finally, we show some numerical experiments confirming the theoretical results, and a couple of applications in view of bifurcation analysis.
\end{abstract}




\section{Introduction}\label{s_introduction}

Including delays within models is often a sound way to describe the relevant phenomena more realistically. Indeed, in various fields of science - such as population dynamics or epidemiology - delays between a cause and the corresponding effects appear rather naturally, which brings the need to resort to delay equations in order to capture adequately the dependence on the past \cite{md86,smith11}.

In many applications, the main interest is towards the dynamical analysis of the  concerned models, including the computation of invariant sets (such as equilibria and periodic solutions) and the study of their asymptotic stability. Regarding periodic solutions, the piecewise orthogonal collocation method to compute those of Renewal Equations (REs) has been first applied in \cite{bdls16}, although the method is not even described therein. Indeed, a formal description appeared first in \cite{aa20} and \cite{ab19} where its validity is only shown by means of some numerical experiments.

The aim of the present paper is to give a detailed illustration of the method and provide a rigorous convergence analysis, having in mind the vast presence of REs in the field of population dynamics \cite{bdgpv12,feller1941,iannelli94,lotka1939}. The convergence analysis follows the main ideas used in \cite{ab20mas} for Retarded Functional Differential Equations (RFDEs), which is in turn based on the abstract approach in \cite{mas15NM}. \cite{ab20mas} represents indeed the solution of the long-standing problem of lack of a rigorous proof of convergence for RFDEs. The contribution of the present paper consists in extending the piecewise collocation and its convergence to REs, by tackling the nontrivial challenges due to the inherent differences concerning the class of equations and the relevant spaces of functions. While some of such challenges are mostly technical, the more consistent ones bring the necessity to resort to the theory of resolvents for Volterra Integral Equations (VIEs, see Appendix \ref{s_appendix}).

We conclude this introduction with Subsection \ref{s_introre}, where we describe the equations of interest and the standard way to formulate the problem of computing periodic solutions as a Boundary Value Problem (BVP). The rest of the paper is divided into three main sections. Section \ref{s_piecewise} describes the piecewise orthogonal collocation method for REs. Section \ref{s_convergence} deals with the theoretical convergence of the method by illustrating how the relevant analysis can be based on the abstract approach in \cite{mas15NM}. An important part of the proof is collected in Appendix \ref{s_appendix}, in order to not interrupt the reading flow. Finally, Section \ref{s_results} shows the results of some numerical experiments on REs from population dynamics, also in view of bifurcation analysis. Python demos are freely available at \url{http://cdlab.uniud.it/software}.

\subsection{Renewal equations and boundary value problems}
\label{s_introre}
In its most general form, an RE can be written as
\begin{equation}\label{eq:RE}
\setlength\arraycolsep{0.1em}\begin{array}{rcl}
x(t)&=&F(x_{t})
\end{array}
\end{equation}
where, for a positive integer $d$, $F:\mathtt{X}\rightarrow\mathbb{R}^{d}$ is autonomous, in general nonlinear and the \emph{state space} $\mathtt{X}$ is a set of functions from $[-\tau,0]$ to $\mathbb{R}^{d}$ for some $\tau>0$, called the \emph{delay}. The \emph{state} of the dynamical system on  $\mathtt{X}$ at time $t$ associated to \eqref{eq:RE} is denoted by $x_{t}$, defined as $x_{t}(\theta):=x(t+\theta),\,\theta\in[-\tau,0]$.
In particular, we develop our analysis for $\mathtt{X}:=B^{\infty}([-\tau, 0],\mathbb{R}^d)$, where $B^{\infty}$ denotes bounded and measurable functions. The elements of $B^{\infty}$ are considered as functions, not as classes of functions that are equal almost everywhere. Such choice, instead of the classical $L^1([-\tau,0],\mathbb{R}^d)$ \cite{diekmann95}, is justified by the need of evaluating the functions pointwise in order to deal with collocation.
A (non-constant) periodic solution of \eqref{eq:RE} with period $\omega>0$  \footnote{ We use the letter $\omega$ to indicate the period, following the notation of \cite{hvl93}.}, if there is any, can be obtained by solving a BVP where the solutions are considered over just one period and the periodicity is imposed to the solution values at the extrema of the period.
Note that this requires to evaluate $x$ at points that fall off the interval $[0,\omega]$, due to the delay. In order to deal with this issue, one can exploit the periodicity to bring back the evaluation to the domain $[0,\omega]$, which, assuming $\omega\geq \tau$\footnote{ A solution with period $\omega$ is also a solution with period $k\omega$ for any positive integer $k$. Moreover, the stability of the relevant hyperbolic orbit is preserved and $t+\theta\geq -\omega$ holds for all $t\in[0,\omega]$.\label{tauleqomega}}, means defining the function $\overline{x}_t\in \mathtt{X}$ as

\begin{equation}\label{perstate}
\overline{x}_t(\theta):=\left\{\setlength\arraycolsep{0.1em}\begin{array}{ll}
x(t+\theta),  &\quad t+\theta\in[0,\omega], \\[1mm]
x(t+\theta+\omega), &\quad t+\theta\in[-\omega,0). \\[1mm]
\end{array}
\right.
\end{equation}
 The periodic BVP can then be formulated as
\begin{equation}\label{barBVPRE}
\left\{\setlength\arraycolsep{0.1em}\begin{array}{rcll}
x(t) &=&  F(\overline{x}_t),&\quad t\in[0,\omega], \\[1mm]
x(0) &=&x(\omega)&\\[1mm]
p(x)&=&0,&
\end{array}
\right.
\end{equation} 
where $p$ is a scalar (usually linear) function defining a so-called \emph{phase condition}, necessary to remove translational invariance. \cite{doe07}. This is the most natural BVP formulation of the problem following the case of RFDEs, e.g., \cite{bad85,bkw06,bz84,elir00,liu94,mas15I,rt74}, as well as the most convenient to consider when developing the numerical method. However, we will introduce in Section \ref{s_convergence} a slightly different, yet equivalent, formulation in view of the convergence analysis based on \cite{mas15NM}.

\bigskip
\noindent In realistic models, such as those describing structured populations, $F$ has usually the form
\begin{equation}\label{hpRHS}
F(\alpha)=\int_{-\tau}^0K(\theta,\alpha(\theta))\dd\theta
\end{equation}
for some integration kernel $K:[-\tau,0]\times\mathbb{R}^d\to\mathbb{R}^d$, or
\begin{equation}\label{altRHS}
F(\alpha)=f\left(\int_{0}^{\tau}k(\sigma)\alpha(-\sigma)\dd\sigma\right),
\end{equation}
for some integration kernel $k:[0,\tau]\to \mathbb{R}^d$ and some function $f:\mathbb{R}^d\to\mathbb{R}^d$. The analysis that follows focuses on \eqref{hpRHS}, even though its validity can be also extended to \eqref{altRHS} (see Remark \ref{r_altRHS} in Section \ref{s_convergence}).
As anticipated previously in the present section, the fundamental differences between the REs we considered and the RFDEs are not limited to their roles in mathematical modeling. Indeed, the theory of REs is typically developed in broader spaces such as $L^1$ \cite{md86,dgg07}, resulting, in principle, in a lower degree of differentiability of the relevant solution.
\section{Piecewise orthogonal collocation}
\label{s_piecewise}
This section describes the numerical method used to compute periodic solutions of \eqref{eq:RE}, starting from a general right-hand side $F$ which, for the moment, is assumed to be computable without resorting to any further numerical approximations.

\bigskip
\noindent Since the period $\omega$ is unknown, it is numerically convenient (see, e.g., \cite{elir00}) to reformulate \eqref{barBVPRE} through the map $s_{\omega}:\mathbb{R}\to\mathbb{R}$ defined by $s_{\omega}(t):=t/\omega$. \eqref{barBVPRE} is thus equivalent to
\begin{equation}\label{rescaledBVPRE}
\left\{\setlength\arraycolsep{0.1em}\begin{array}{rcll}
x(t) &=&  F(\overline{x}_t\circ s_{\omega}),&\quad t\in[0,1], \\[1mm]
x(0) &=&x(1)&\\[1mm]
p(x)&=&0,&
\end{array}
\right.
\end{equation}
the solution of which is intended to be defined in $[0,1]$.

\bigskip
\eqref{rescaledBVPRE} can be solved numerically through \emph{piecewise} orthogonal collocation \cite{elir00}. This is particularly useful when adaptive meshes, to better follow the solution profile, might be needed, and is now a standard approach (originally developed for ODEs \cite{amr88}, see \texttt{MatCont} \cite{matcont}). It means looking for an $m$-degree piecewise continuous polynomial $u$ in $[0,1]$ and $w\in\mathbb{R}$ that satisfy the following system having dimension $(1+Lm)\times d+1$:
\begin{equation*}
\left\{\setlength\arraycolsep{0.1em}\begin{array}{rcll}
u(t_{i,j}) &=& F(\overline{u}_{t_{i,j}}\circ s_w),&\quad j\in\{1,\ldots,m\},\quad i\in\{1,\ldots,L\}, \\[1mm]
u(0) &=&u(1)&\\[1mm]
p(u)&=&0&
\end{array}
\right.
\end{equation*}
for a given mesh $0= t_0<\cdots<t_L=1$ and collocation points $\{t_{i,j}\}_{i,j}$ such that $t_{i-1}  < t_{i,1}<\cdots<t_{i,m} < t_{i}$ for all $i\in\{1,\ldots,L\}$. The unknowns are, other than $w$, those of the form $u_{i,j}:=u(t_{i,j})$  for $(i,j)=(1,0)$ and $i\in\{1,\ldots,L\}$, $j\in\{1,\ldots,m\}$\footnote{In fact, one could also consider to represent $u$ through its values at other sets of nodes as unknowns, not necessarily $\{t_{i,j}\}_{i,j}$ (see Remark \ref{r_representation} at the end of Section \ref{s_convergence}).}.
\begin{remark}
Typically, periodic solutions are computed within a continuation framework.  This provides a reasonable choice for the phase condition, necessary to ensure the actual well-posedeness of \eqref{rescaledBVPRE}.
\end{remark}
\bigskip
As mentioned at the end of Subsection \ref{s_introre}, in applications from population dynamics right-hand sides usually feature an integral, therefore cannot be exactly computed in general. This is also the case of \eqref{hpRHS}, which reads, once the time has been rescaled,
\begin{equation}\label{hprescaled}
F(x_t\circ s_{\omega})=\int_{-\frac{\tau}{\omega}}^0\omega K(\omega\theta,x_t(s_{\omega}(\omega\theta)))\dd\theta=\omega\int_{-\frac{\tau}{\omega}}^0 K(\omega\theta,x(t+\theta))\dd\theta,
\end{equation}
where now $x_t\in X:=B^{\infty}([-1,0],\mathbb{R}^d)$. Observe that, although the corresponding natural state space is in fact a Banach space of functions defined in $[-\tau/\omega,0]$, one could choose spaces of functions defined in $[-r,0]$ for any $1\geq r\geq\tau/\omega$\footnote{ The extension of the state space to $[-1,0]$ is necessary since $\omega$ varies within an iterative procedure to find a numerical solution while the space needs to be fixed, as is required for the forthcoming analysis. Observe that such extension does not affect the dynamics: indeed, initial states that only differ in $[-1,-\tau/\omega]$ lead to the same orbits, but different orbits cannot cross.}.

\bigskip
Assuming that $K$ can be exactly computed, which is usually the case in applications, the approximation of \eqref{hprescaled} through quadrature as
\begin{equation*}
F_M(x_t\circ s_{\omega}):=\omega \sum_{i=0}^Mw_i K(\omega\eta_i,x(t+\eta_i)), 
\end{equation*}
where $M$ is a given approximation level and $-\tau/\omega\leq\eta_0<\cdots<\eta_M\leq 0$, can also be exactly computed. Such an approximation corresponds to the \emph{secondary discretization} introduced in Subsection \ref{sub_numerical} and used in the convergence analysis that follows. The nodes $\eta_0,\ldots,\eta_M$ and the corresponding weights $w_0,\ldots,w_M$ are meant to define a suitable quadrature formula by exploiting possible irregularities in $K$, meaning that their choice does not need to be made a priori. Moreover, note that the quadrature nodes vary together with $\omega$, since the latter is unknown. In particular, they are completely independent of the collocation nodes mentioned earlier.
\section{Convergence analysis}
\label{s_convergence}
This section describes the theoretical convergence analysis of the numerical method described in Section \ref{s_piecewise}, following the abstract approach \cite{mas15NM}. In particular, the convergence analysis that follows applies to the Finite Element Method (FEM), which consists in letting $L\to\infty$ while keeping $m$ fixed and is the classical approach considered in practical applications (e.g., in \texttt{MatCont} \cite{matcont} or some versions of \texttt{DDE-Biftool} \cite{ddebiftool,sie15}).
A few comments on the Spectral Element Method (SEM, $m\to\infty$ while keeping $L$ fixed) will follow in Subsection \ref{s_numerical}.

\bigskip
\noindent Following the approach for RFDEs in \cite{ab20mas}, we first reformulate \eqref{rescaledBVPRE} as
\begin{equation}\label{convBVPRE}
\left\{\setlength\arraycolsep{0.1em}\begin{array}{rcll}
x(t) &=& F(x_t\circ s_{\omega}),&\quad t\in[0,1], \\[1mm]
x_0 &=& x_1&\\[1mm]
p(x\vert_{[0,1]})&=&0,&
\end{array}
\right.
\end{equation}
i.e., by imposing the periodicity condition to the states at the extrema of the period rather than to the solution values. In this case the solution $x$ is intended as a map defined in $[-1,1]$ and there is no need to resort to \eqref{perstate}.

\bigskip
Although formulations \eqref{rescaledBVPRE} and \eqref{convBVPRE} are formally different, they are mathematically equivalent and also lead to fundamentally equivalent numerical methods. Indeed, when applying the numerical method described in Section \ref{s_piecewise} to the problem \eqref{convBVPRE} one just introduces redundant variables\footnote{A convergence analysis based on \eqref{rescaledBVPRE} can still be obtained by following the structure of the one presented here, see \cite{adena21}.}.

\bigskip
The second step consists in observing that \eqref{convBVPRE} fits into the general form  of the BVP addressed in \cite{mas15NM}:
\begin{equation}\label{generalBVP}
\begin{cases}
u=\mathcal{F}(\mathcal{G}(u,\alpha),u,\beta)\\
\mathcal{B}(\mathcal{G}(u,\alpha),u,\beta)=0.
\end{cases}
\end{equation}
Here the relevant solution $v:=\mathcal{G}(u,\alpha)$ is assumed to lie in a normed space $\mathbb{V}$ of functions $[-1,1]\to\mathbb{R}^d$, $u$ in a Banach space $\mathbb{U}$ of functions $[0,1]\to\mathbb{R}^d$, and the operator $\mathcal{G}:\mathbb{U}\times\mathbb{A}\rightarrow\mathbb{V}$ represents a linear operator which reconstructs the solution given $u$ and an initial state $\alpha$, also lying in a Banach space $\mathbb{A}$ of functions $[-1,0]\to\mathbb{R}^d$. $\beta$ is a vector of parameters living in a Banach space $\mathbb{B}$. 
The first line of \eqref{generalBVP} represents the concerned functional equation via the function $\mathcal{F}:\mathbb{V}\times\mathbb{U}\times\mathbb{B}\rightarrow\mathbb{U}$ and the second represents the boundary conditions via $\mathcal{B}:\mathbb{V}\times\mathbb{U}\times\mathbb{B}\rightarrow\mathbb{A}\times\mathbb{B}$. \footnote{Note that \cite{mas15NM} concerns RFDEs of neutral type, where $u$ plays the role of the derivative of the solution $v$. In the case of REs, however, no derivatives are involved and $\mathbb{U}$ can play the role of the space of the solution in $[0,1]$.}
\bigskip
In \cite{mas15NM}, \eqref{generalBVP} is then translated into a fixed-point problem, the so-called {\it Problem in Abstract Form} (PAF) which consists in finding $(v^{\ast},\beta^{\ast})\in \mathbb{V}\times\mathbb{B}$ with $v^{\ast}:=\mathcal{G}(u^{\ast},\alpha^{\ast})$ and $(u^{\ast},\alpha^{\ast},\beta^{\ast})\in \mathbb{U}\times\mathbb{A}\times\mathbb{B}$ such that 
\begin{equation}\label{PAF}
(u^{\ast},\alpha^{\ast},\beta^{\ast})=\Phi(u^{\ast},\alpha^{\ast},\beta^{\ast})
\end{equation}
for $\Phi:\mathbb{U}\times\mathbb{A}\times\mathbb{B}\rightarrow\mathbb{U}\times\mathbb{A}\times\mathbb{B}$ given by
\begin{equation}\label{Phi}
\Phi(u,\alpha,\beta):=
\begin{pmatrix}
\mathcal{F}(\mathcal{G}(u,\alpha),u,\beta)\\[2mm]
(\alpha,\beta)-\mathcal{B}(\mathcal{G}(u,\alpha),u,\beta)
\end{pmatrix}.
\end{equation}
In the sequel we always use the superscript $^{\ast}$ to denote quantities relevant to fixed points (which correspond to non-constant periodic solutions).

\bigskip
It follows that \eqref{convBVPRE} leads to an instance of \eqref{Phi} with $\mathcal{G}$, $\mathcal{F}$ and $\mathcal{B}$ given respectively by
\begin{equation}\label{G2}
\mathcal{G}_{}(u,\alpha)(t):=\begin{cases}
\displaystyle u(t),&t\in(0,1],\\[2mm]
\alpha(t),&t\in[-1,0],
\end{cases}
\end{equation}
\begin{equation}\label{F2}
\mathcal{F}_{}(v,u,\omega):= F(v_{(\cdot)}\circ s_{\omega})
\end{equation}
and
\begin{equation}\label{B2}
\mathcal{B}_{}(v,u,\omega):=
\begin{pmatrix}
v_{0}-v_{1}\\[2mm]
p(u)
\end{pmatrix}.
\end{equation}
The boundary operator is linear and independent of $\omega$.

\bigskip
The fact that \eqref{convBVPRE} can be rewritten as a PAF does not imply that the the convergence framework in \cite{mas15NM} can be applied either way. In fact, several assumptions are required. These include theoretical assmuptions, the validity of which depends on the choices of the spaces, as well as on the regularity of the integrand $K$ in the right-hand side \eqref{hpRHS}. Subsection \ref{sub_theoretical} includes the definitions of such assumptions and their statements as propositions, instanced according to the problems of interest. The other assumptions required concern instead the reduction of the problem to a finite-dimensional one, and will be dealt with similarly in Subsection \ref{sub_numerical}. Concerning the proofs of such propositions,  we will go through the main points, focusing on the differences with respect to the analogous propositions in the RFDE case \cite{ab20mas}, with the exception of a more complicated one to which we dedicate the entire Appendix \ref{s_appendix}.
\subsection{Theoretical assumptions}
\label{sub_theoretical}
The hypotheses on the original problem needed to prove the validity of the theoretical assumptions in \cite{mas15NM} are:
\begin{itemize}
\item[(T1)] $\mathtt{X}=B^{\infty}([-\tau,0],\mathbb{R}^{d})$, $X=B^{\infty}([-1,0],\mathbb{R}^{d})$;
\item[(T2)] $\mathbb{U}_{}=B^{\infty}([0,1],\mathbb{R}^{d})$, $\mathbb{V}_{}=B^{\infty}([-1,1],\mathbb{R}^{d})$, $\mathbb{A}_{}=B^{\infty}([-1,0],\mathbb{R}^{d})$;
\item[(T3)] $K:\mathbb{R}\times\mathbb{R}^d\to\mathbb{R}^{d}$ is piecewise continuous and has partial derivatives $D_1K,\,D_2K$ which are measurable with respect to both their arguments;
\item[(T4)] the map $x\mapsto D_2K(r,x)$ is piecewise continuous for all $r\in \mathbb{R}$;
\item[(T5)] there exist $r>0$ and $\kappa\geq0$ such that
\begin{equation*}
\left\{\setlength\arraycolsep{0.1em}\begin{array}{rcl}
\|D_1K(\omega\cdot,v_t)-D_1K(\omega^*\cdot,v^{\ast}_{t})\|_{\mathbb{R}^{d}\leftarrow \mathbb{R}}\leq\kappa\|(v_t,\omega)-(v^{\ast}_{t}, \omega^{\ast})\|_{X\times\mathbb{R}}\\
\|D_2K(\omega\cdot,v_t)-D_2K(\omega^*\cdot,v^{\ast}_{t})\|_{\mathbb{R}^{d}\leftarrow \mathbb{R}}\leq\kappa\|(v_t,\omega)-(v^{\ast}_{t}, \omega^{\ast})\|_{X\times\mathbb{R}}
\end{array}\right.
\end{equation*}
for every $(v_t,\omega)\in b((v^{\ast}_{t},\omega^{\ast}),r)$\footnote{$b(c,r)$ denotes the closed ball centered in $c$ having radius $r$.}, uniformly with respect to $t\in[0,1]$.
\end{itemize}
\bigskip
The first theoretical assumption (A$\mathfrak{F}\mathfrak{B}$, \cite[page 534]{mas15NM}) concerns the Fr\'echet-differentiability of the operators $\mathcal{F}$ and $\mathcal{B}$ appearing in \eqref{Phi}. Since $p$ is linear, so is $\mathcal{B}$ in \eqref{B2}, hence the latter is Fr\'echet-differentiable. The validity of the assumption is thus a direct consequence of the following.

\begin{proposition}\label{p_Afb}
Under (T1), (T2) and (T3), $\mathcal{F}_{}$ in \eqref{F2} is Fr\'echet-differentiable,  from the right with respect to $\omega$, at every $(\hat v,\hat u,\hat\omega)\in\mathbb{V}_{}\times\mathbb{U}_{}\times(0,+\infty)$ and
\begin{equation}\label{DF2}
D\mathcal{F}_{}(\hat v,\hat u,\hat\omega)(v,u,\omega)=\mathfrak{L}_{}(\cdot;\hat v,\hat\omega)[v_{(\cdot)}\circ s_{\hat\omega}]+\omega\mathfrak{M}_{}(\cdot;\hat v,\hat\omega)
\end{equation}
for $(v,u,\omega)\in\mathbb{V}_{}\times\mathbb{U}_{}\times(0,+\infty)$, where, for $t\in[0,1]$,
\begin{equation}\label{L2}
\mathfrak{L}_{}(t;\hat v,\hat\omega)[v_t\circ s_{\hat\omega}]:=\hat\omega\int_{-\frac{\tau}{\hat\omega}}^0D_2K(\hat\omega\theta, \hat v(t+\theta))v(t+\theta)\dd\theta
\end{equation}
and
\begin{equation}\label{M2}
\setlength\arraycolsep{0.1em}\begin{array}{rl}
\mathfrak{M}_{}(t;v,\omega):=&\displaystyle\int_{-\frac{\tau}{\omega}}^0 K(\omega\theta, v(t+\theta))\dd\theta-\displaystyle\frac{\tau}{\omega} K\left(-\tau,v\left(t-\frac{\tau}{\omega}\right)\right)\\[2mm]
+&\displaystyle\omega\int_{-\frac{\tau}{\omega}}^0D_1K(\omega\theta,v(t+\theta))\theta\dd\theta.
\end{array}
\end{equation}
\end{proposition}
\begin{proof} 
The proof is rather technical and goes as that of \cite[Proposition 2.1]{ab20mas}, therefore we avoid to repeat all the steps for brevity and to better concentrate on the differences. Basically, the expression \eqref{DF2}, defined through \eqref{L2} and \eqref{M2}, is directly proven to satisfy the definition of differentiable function according to \cite[Definition 1.1.1]{ampr95}. It is worth pointing out that assuming an integral right-hand side such as \eqref{hpRHS}, which is anyway typical in applications from population dynamics, is crucial for this proposition in the case of REs. Basically, for the thesis to hold it is required that the right-hand side always lies in a more regular space than $\mathbb{U}$ (which is always the case for RFDEs, where $\mathbb{U}$ plays the role of the space of the derivatives). This can be observed by looking at the last addend of \cite[(2.12)]{ab20mas}, where the derivative of the state of an element of $\mathbb{V}$ appears as a factor. Without any assumption whatsoever on $F$ the same would happen in Proposition \ref{p_Afb}, which would be a problem since $\mathbb{V}$ is as regular as $\mathbb{U}$ in the present case.\end{proof}

\bigskip
The second theoretical assumption (A$\mathfrak{G}$, \cite[page 534]{mas15NM}) concerns the boundedness of $\mathcal{G}$ defined in \eqref{G2}. The following proposition concerns its validity, and its proof is an immediate consequence of the definition \eqref{G2}.

\begin{proposition}\label{l_G}
Under (T2), $\mathcal{G}_{}$ is bounded.
\end{proposition}

\bigskip
\noindent The third theoretical assumption (A$x^*1$, \cite[page 536]{mas15NM}) concerns the local Lipschitz continuity of the Fr\'echet derivative of the fixed point operator $\Phi$ in \eqref{Phi} at the relevant fixed points. In the sequel $(u^{\ast},\alpha^{\ast},\omega^{\ast})\in\mathbb{U}_{}\times\mathbb{A}_{}\times(0,+\infty)$ is a fixed point of $\Phi_{}$ and $x^{\ast}$ is the corresponding $1$-periodic solution of \eqref{eq:RE}. With respect to the validity of Assumption A$x^*1$ the following holds.

\begin{proposition}\label{p_Ax*1}
Under (T1), (T2), (T3) and (T5), there exist $r_{}\in(0,\omega^{\ast})$ and $\kappa_{}\geq0$ such that
\begin{equation*}
\setlength\arraycolsep{0.1em}\begin{array}{rcl}
\|D\Phi_{}(u,\alpha,\omega)&-&D\Phi_{}(u^{\ast},\alpha^{\ast},\omega^{\ast})\|_{\mathbb{U}_{}\times\mathbb{A}_{}\times\mathbb{R}\leftarrow\mathbb{U}_{}\times\mathbb{A}_{}\times(0,+\infty)}\\[2mm]
&\leq&\kappa_{}\|(u,\alpha,\omega)-(u^{\ast},\alpha^{\ast},\omega^{\ast})\|_{\mathbb{U}_{}\times\mathbb{A}_{}\times\mathbb{R}}
\end{array}
\end{equation*}
for all $(u,\alpha,\omega)\in b((u^{\ast},\alpha^{\ast},\omega^{\ast}),r_{})$.
\end{proposition} 
\begin{proof} 
Just as its RFDE counterpart in \cite{ab20mas}, the proposition can be proved thanks to the fact that  $u^*$ lies in fact in a more regular subspace of its space $\mathbb{U}$, which is a consequence of the assumption \eqref{hpRHS}.\end{proof}

The fourth (and last) theoretical assumption (A$x^{\ast}2$, \cite[page 536]{mas15NM}), concerns the well-posedness of a linearized inhomogeneous version of the PAF \eqref{PAF}. Its validity can be proved under (T1), (T2), (T3) and (T4), together with an additional requirement, which in turn follows from assuming, e.g., the {\it hyperbolicity} of the periodic solution\footnote{Let us remark that the condition of hyperbolicity is necessary for the local stability analysis of periodic solutions in view of the Principle of Linearized Stability \cite{bl20}.} of the original problem.
It is convenient to introduce the abbreviations
\begin{equation}\label{L*M*}
\mathfrak{L}_{}^{\ast}:=\mathfrak{L}_{}(\cdot;v^{\ast},\omega^{\ast}),\qquad\mathfrak{M}_{}^{\ast}:=\mathfrak{M}_{}(\cdot;v^{\ast},\omega^{\ast}).
\end{equation}
\begin{proposition}\label{p_Ax*2}
Under (T1), (T2), (T3) and (T4), let $T^*(t,s):X\to X$ be the evolution operator for the linear homogeneous RE
$$
x(t)=\mathfrak{L}^{\ast}_{}(t)[x_t\circ s_{\omega^{\ast}}].
$$
If $1\in\sigma(T_{}^{\ast}(1,0))$ is simple, then the linear bounded operator $I_{\mathbb{U}_{}\times\mathbb{A}_{}\times\mathbb{B}_{}}-D\Phi_{}(u^{\ast},\alpha^{\ast},\omega^{\ast})$ is invertible, i.e., for all ${\color{black}(\tilde u,\tilde\alpha,\tilde\omega)}\in\mathbb{U}_{}\times\mathbb{A}_{}\times\mathbb{B}_{}$ there exists a unique $(u,\alpha,\omega)\in\mathbb{U}_{}\times\mathbb{A}_{}\times\mathbb{B}_{}$ such that
\begin{equation}\label{Ax*20}
\left\{\setlength\arraycolsep{0.1em}\begin{array}{l}
u=\mathfrak{L}_{}^{\ast}[\mathcal{G}_{}(u,\alpha)_{(\cdot)}\circ s_{\omega^{\ast}}]+\omega\mathfrak{M}_{}^{\ast}+{\color{black}\tilde u}\\[2mm]
\alpha=u_{1}+{\color{black}\tilde\alpha}\\[2mm]
p(u)={\color{black}\tilde\omega}.
\end{array}
\right.
\end{equation}
\end{proposition} 
\begin{proof}
\eqref{Ax*20} can be treated as an initial value problem for $v=\mathcal{G}(u,\alpha)$, i.e.,
\begin{equation}\label{Ax*211}
\left\{\setlength\arraycolsep{0.1em}\begin{array}{l}
v(t)=\mathfrak{L}^*(t)[v_{t}\circ s_{\omega^*}]+\omega\mathfrak{M}^*(t)+{\color{black}\tilde u}(t)\\[2mm]
v_{0}=\alpha
\end{array}
\right.
\end{equation}
for $t\in[0,1]$, imposing then the boundary conditions in \eqref{Ax*20}. We can write $v(t)=v^{(1)}(t)+v^{(2)}(t)$, where $v^{(1)}(t)$ is the solution of
\begin{equation*}
\left\{\setlength\arraycolsep{0.1em}\begin{array}{l}
v^{(1)}(t)=\mathfrak{L}^{\ast}(t)[v_{t}\circ s_{\omega^{\ast}}]\\[2mm]
v^{(1)}_{0}=\alpha,
\end{array}
\right.
\end{equation*}
which means that $v^{(1)}_t=T^*(t,0)\alpha$, while $v^{(2)}(t)$ is the solution of
\begin{equation*}
\left\{\setlength\arraycolsep{0.1em}\begin{array}{l}
v^{(2)}(t)=\omega\mathfrak{M}^{\ast}(t)+u_{0}(t)\\[2mm]
v^{(2)}_{0}=0,
\end{array}
\right.
\end{equation*}
i.e., $v^{(2)}_t=\omega\mathfrak{M}_t^{*(0)}+{\color{black}\tilde u_{t}}^{(0)}$ where, in turn,
\begin{equation*}
\mathfrak{M}_t^{*(0)}(\theta):=\begin{cases}
0,&t+\theta\in[-1,0],\\
\mathfrak{M}^{\ast}(t+\theta),&t+\theta\in(0,1]
\end{cases}
\end{equation*}
and
\begin{equation*}
{\color{black}\tilde u_{t}}^{(0)}(\theta):=\begin{cases}
0,&t+\theta\in[-1,0],\\
{\color{black}\tilde u}(t+\theta),&t+\theta\in(0,1].
\end{cases}
\end{equation*}
The first boundary condition in \eqref{Ax*20} gives then
\begin{equation}\label{Ax*21}
\alpha=T^{\ast}(1,0)\alpha+\omega\mathfrak{M}_1^{*(0)}+{\color{black}\tilde u_{1}}^{(0)}+\alpha_0.
\end{equation}
Now, the proof can be concluded as that of \cite[Proposition 2.7]{ab20mas}, by defining the elements $\xi_{1}^{\ast}:=\mathfrak{M}_1^{*(0)}$ and $\xi_{2}^{\ast}:={\color{black}\tilde u_{1}}^{(0)}+\alpha_0$, and assuming $p(v(\cdot;\alpha)\vert_{[0,1]})\neq 0$ (see \cite[Remark 2.8]{ab20mas}) and $\xi_1^*\not\in R$, where $R$ is the range of the operator $I_X-T^*(1,0)$ (see Appendix \ref{s_appendix} for a detailed proof of the latter point).
\end{proof}

\subsection{Numerical assumptions}
\label{sub_numerical}
As anticipated, the present subsection deals with the numerical assumptions, which concern the chosen discretization scheme for the numerical method. Such scheme is defined by the {\it primary} and the {\it secondary} discretizations.

\bigskip
As in \cite{ab20mas}, the primary discretization consists in reducing the spaces $\mathbb{U}$ and $\mathbb{A}$ to finite-dimensional spaces $\mathbb{U}_{L}$ and $\mathbb{A}_{L}$, given a level of discretization $L$. This happens by means of {\it restriction} operators $\rho_{L}^{+}:\mathbb{U}\rightarrow\mathbb{U}_{L}$, $\rho_{L}^{-}:\mathbb{A}\rightarrow\mathbb{A}_{L}$ and {\it prolongation} operators $\pi_{L}^{+}:\mathbb{U}_{L}\rightarrow\mathbb{U}$, $\pi_{L}^{-}:\mathbb{A}_{L}\rightarrow\mathbb{A}$, which extend respectively to $R_{L}:\mathbb{U}\times\mathbb{A}\times\mathbb{B}\rightarrow\mathbb{U}_{L}\times\mathbb{A}_{L}\times\mathbb{B}$ given by $R_{L}(u,\alpha,\omega):=(\rho_{L}^{+}u,\rho_{L}^{-}\alpha,\omega)$ and $P_{L}:\mathbb{U}_{L}\times\mathbb{A}_{L}\times\mathbb{B}\rightarrow\mathbb{U}\times\mathbb{A}\times\mathbb{B}$ given by $P_{L}(u_{L},\alpha_{L},\omega):=(\pi_{L}^{+}u_{L},\pi_{L}^{-}\alpha_{L},\omega)$.
All of them are linear and bounded. In the following we describe the specific choices we make in this context, based on piecewise polynomial interpolation.

Starting from $\mathbb{U}$, which concerns the interval $[0,1]$, we choose the uniform {\it outer} mesh
\begin{equation}\label{outmesh+}
\Omega_{L}^{+}:=\{t_{i}^{+}=ih\ :\ i\in\{0,\ldots,L\},\,h=1/L\}\subset[0,1],
\end{equation}
and {\it inner} meshes
\begin{equation}\label{inmesh+}
\Omega_{L,i}^{+}:=\{t_{i,j}^{+}:=t_{i-1}^{+}+c_{j}h\ :\ j\in\{1,\ldots,m\}\}\subset[t_{i-1}^{+},t_{i}^{+}],\quad i\in\{1,\ldots,L\},
\end{equation}
where $0<c_{1}<\cdots<c_{m}<1$ are given abscissae for $m$ a positive integer. Correspondingly, we define
\begin{equation}\label{UL}
\mathbb{U}_{L}:=\mathbb{R}^{(1+Lm)\times d},
\end{equation}
whose elements $u_{L}$ are indexed as
\begin{equation}\label{uL}
u_{L}:=(u_{1,0},u_{1,1},\ldots,u_{1,m},\ldots,u_{L,1}\ldots,u_{L,m})^{T}
\end{equation}
with components in $\mathbb{R}^{d}$. Finally, we define, for $u\in\mathbb{U}$,
\begin{equation}\label{rL+}
\rho_{L}^{+}u:=(u(0),u(t_{1,1}^{+}),\ldots,u(t_{1,m}^{+}),\ldots,u(t_{L,1}^{+})\ldots,u(t_{L,m}^{+}))^{T}\in\mathbb{U}_{L}
\end{equation}
and, for $u_{L}\in\mathbb{U}_{L}$, $\pi_{L}^{+}u_{L}\in\mathbb{U}$ as the unique element of the space
\begin{equation}\label{PiLm+}
\Pi_{L,m}^{+}:=\{p\in C([0,1],\mathbb{R}^d)\ :\ p\vert_{[t_{i-1}^{+},t_{i}^{+}]}\in\Pi_m,\;i\in\{1,\ldots,L\}\}
\end{equation}
such that
\begin{equation}\label{pL+}
\pi_{L}^{+}u_{L}(0)=u_{1,0},\quad\pi_{L}^{+}u_{L}(t_{i,j}^{+})=u_{i,j},\quad j\in\{1,\ldots,m\},\;i\in\{1,\ldots,L\}.
\end{equation}
Above $\Pi_{m}$ is the space of $\mathbb{R}^d$-valued polynomials having degree $m$ and, when needed, we represent $p\in\Pi_{L,m}^{+}$ through its pieces as
\begin{equation}\label{lagrange}
p\vert_{[t_{i-1}^{+},t_{i}^{+}]}(t)=\sum_{j=0}^{m}\ell_{m,i,j}(t)p(t_{i,j}^{+}),\quad t\in[0,1],
\end{equation}
where, for ease of notation, we implicitly set
\begin{equation}\label{t0+}
t_{i,0}^{+}:=t_{i-1}^{+},\quad i\in\{1,\ldots,L\},
\end{equation}
and $\{\ell_{m,i,0},\ell_{m,i,1},\ldots,\ell_{m,i,m}\}$ is the Lagrange basis relevant to the nodes $\{t_{i,0}^{+}\}\cup\Omega_{L,i}^{+}$. Observe that the latter is invariant with respect to $i$ as long as we fix the abscissae $c_{j}$, $j\in\{1,\ldots,m\}$, defining the inner meshes \eqref{inmesh+}. Indeed, for every $i\in\{1,\ldots,L\}$,
\begin{equation*}
\ell_{m,i,j}(t)=\ell_{m,j}\left(\frac{t-t_{i-1}^{+}}{h}\right),\quad t\in[t_{i-1}^{+},t_{i}^{+}],
\end{equation*}
where $\{\ell_{m,0},\ell_{m,1},\ldots,\ell_{m,m}\}$ is the Lagrange basis in $[0,1]$ relevant to the abscissae $c_{0},c_{1},\ldots,c_{m}$ with $c_{0}:=0$. 

Similarly, for $\mathbb{A}$, which concerns the interval $[-1,0]$, we choose
\begin{equation}\label{outmesh-}
\Omega_{L}^{-}:=\{t_{i}^{-}=ih-1\ :\ i\in\{0,\ldots,L\},\;h=1/L\}\subset[-1,0],
\end{equation}
and
\begin{equation}\label{inmesh-}
\Omega_{L,i}^{-}:=\{t_{i,j}^{-}:=t_{i-1}^{-}+c_{j}h\ :\ j\in\{1,\ldots,m\}\}\subset[t_{i-1}^{-},t_{i}^{-}],\quad i\in\{1,\ldots,L\}.
\end{equation}
Correspondingly, we define
\begin{equation}\label{AL}
\mathbb{A}_{L}:=\mathbb{R}^{(1+Lm)\times d}
\end{equation}
with indexing
\begin{equation}\label{psiL}
\alpha_{L}:=(\alpha_{1,0},\alpha_{1,1},\ldots,\alpha_{1,m},\ldots,\alpha_{L,1}\ldots,\alpha_{L,m})^{T};
\end{equation}
for $\alpha\in\mathbb{A}$,
\begin{equation}\label{rL-}
\rho_{L}^{-}\alpha:=(\alpha(-1),\alpha(t_{1,1}^{-}),\ldots,\alpha(t_{1,m}^{-}),\ldots,\alpha(t_{L,1}^{-})\ldots,\alpha(t_{L,m}^{-}))^{T}\in\mathbb{A}_{L}
\end{equation}
and, for $\alpha_{L}\in\mathbb{A}_{L}$, $\pi_{L}^{-}\alpha_{L}\in\mathbb{A}$ as the unique element of the space
\begin{equation}\label{PiLm-}
\Pi_{L,m}^{-}:=\{p\in C([-1,0],\mathbb{R}^d)\ :\ p\vert_{[t_{i-1}^{-},t_{i}^{-}]}\in\Pi_m,\;i\in\{1,\ldots,L\}\}
\end{equation}
such that
\begin{equation}\label{pL-}
\pi_{L}^{-}\alpha_{L}(-1)=\alpha_{1,0},\quad\pi_{L}^{-}\alpha_{L}(t_{i,j}^{-})=\alpha_{i,j},\quad j\in\{1,\ldots,m\},\;i\in\{1,\ldots,L\}.
\end{equation}
Elements in $\Pi_{L,m}^{-}$ are represented in the same way as those of $\Pi_{L,m}^{+}$ by suitably adapting both \eqref{lagrange} and \eqref{t0+}.
\begin{remark}\label{r_primary}
It is worth pointing out that more general choices can be made concerning outer and inner meshes. In particular, as already remarked in Section \ref{s_piecewise}, in practical applications {\it adaptive} outer meshes represent a standard for RFDEs, see, e.g., \cite{elir00}. As for inner meshes, abscissae including the extrema of $[0,1]$ can also be considered, paying attention to put the correct constraints at the internal outer nodes, i.e., $t_{i}^{\pm}$ for $i\in\{1,\ldots,L-1\}$.
\end{remark}
\bigskip
The secondary discretization consists in replacing $\mathcal{F}$ in the first of \eqref{Phi} with an operator $\mathcal{F}_{M}$ that can be exactly computed, for a given level of discretization $M$. In particular, we define $\mathcal{F}_{M}$ through an approximated version $F_{M}$ of the right-hand side $F$ defined in \eqref{hpRHS} as 
\begin{equation}\label{quadrule}
\mathcal{F}_{M}(u,\alpha,\omega)=F_M(\mathcal{G}(u,\alpha)_{(\cdot)}\circ s_{\omega}):=\omega \sum_{i=0}^Mw_i K(\omega\eta_i,\mathcal{G}(u,\alpha)_{\eta_i}), 
\end{equation}
where $-\tau/\omega\leq\eta_0<\cdots<\eta_M\leq 0$. Indeed, in realistic applications the integrand function in \eqref{hpRHS} can be exactly computed, as already remarked at the end of Section \ref{s_piecewise}. Correspondingly, $\Phi_{M}$ is the operator obtained by replacing $\mathcal{F}$ in $\Phi$ in \eqref{Phi} with its approximated version, i.e., $\Phi_{M}:\mathbb{U}\times\mathbb{A}\times\mathbb{B}\rightarrow\mathbb{U}\times\mathbb{A}\times\mathbb{B}$ defined by
\begin{equation}\label{PhiM}
\Phi_{M}(u,\alpha,\omega):=
\begin{pmatrix}
F_{M}(\mathcal{G}(u,\alpha)_{(\cdot)}\circ s_{\omega})\\[2mm]
u_{1}\\[2mm]
\omega-p(u)
\end{pmatrix}.
\end{equation}
A secondary discretization for $\mathcal{G}$ in \eqref{Phi} is instead unnecessary, since it can be evaluated exactly in $\pi_{L}^{+}\mathbb{U}_{L}\times\pi_{L}^{-}\mathbb{A}_{L}$ according to \eqref{UL} and \eqref{AL}. As for the operator $p$ defining the phase condition in \eqref{Phi}, we assume that it can be evaluated exactly in $\pi_{L}^{+}\mathbb{U}_{L}$. \footnote{This is indeed true in the case of integral phase conditions if the piecewise quadrature is based on the mesh of the primary discretization, which is the standard approach in practical applications.}

\bigskip
From the two discretizations together we can define the discrete version 
\begin{equation*}
\Phi_{L,M}:=R_{L}\Phi_{M} P_{L}:\mathbb{U}_{L}\times\mathbb{A}_{L}\times\mathbb{B}\rightarrow\mathbb{U}_{L}\times\mathbb{A}_{L}\times\mathbb{B}
\end{equation*}
of the fixed point operator $\Phi$ in \eqref{Phi} as
\begin{equation*}
\Phi_{L,M}(u_{L},\alpha_{L},\omega):=
\begin{pmatrix}
\rho_{L}^{+}F_{M}(\mathcal{G}(\pi_{L}^{+}u_{L},\pi_{L}^{-}\alpha_{L})_{(\cdot)}\circ s_{\omega})\\[2mm]
\rho_{L}^{-}(\pi_{L}^{+}u_{L})_1\\[2mm]
\omega-p(\pi_{L}^{+}u_{L})
\end{pmatrix}.
\end{equation*}
A fixed point $(u_{L,M}^{\ast},\alpha_{L,M}^{\ast},\omega_{L,M}^{\ast})$ of $\Phi_{L,M}$ can be found by standard solvers for nonlinear systems of algebraic equations and, as will be shown in Subsection \ref{s_convresult}, its prolongation $P_{L}(u_{L,M}^{\ast},\alpha_{L,M}^{\ast},\omega_{L,M}^{\ast})$ is then considered as an approximation of a fixed point $(u^{\ast},\alpha^{\ast},\omega^{\ast})$ of $\Phi$ in \eqref{Phi}. Correspondingly, $v_{L,M}^{\ast}:=\mathcal{G}(\pi_{L}^{+}u_{L,M}^{\ast},\pi_{L}^{-}\alpha_{L,M}^{\ast})$ is considered as an approximation of the solution $v^{\ast}=\mathcal{G}(u^{\ast},\alpha^{\ast})$ of \eqref{Phi}.

\bigskip
The hypotheses on the discretization method needed to prove the validity of the numerical assumptions in \cite{mas15NM} are:
\begin{itemize}
\item[(N1)] the primary discretization of the space $\mathbb{U}$ is based on the choices \eqref{outmesh+}--\eqref{pL+};
\item[(N2)] the primary discretization of the space $\mathbb{A}$ is based on the choices \eqref{outmesh-}--\eqref{pL-};
\item[(N3)] the nodes $\eta_0,\ldots,\eta_M$, together with the weights $w_0,\ldots,w_M$ chosen for the secondary discretization as in \eqref{quadrule} define an interpolatory quadrature formula which is convergent in $B^{\infty}([0,1],\mathbb{R}^d)$.
\end{itemize}

\bigskip
The first numerical assumption to be verified in \cite{mas15NM} is Assumption A$\mathfrak{F}_K\mathfrak{B}_K$ (page 535). As already observed, $\mathcal{B}$ and $p$ are linear functions, thus the proof of its validity is a direct consequence of the following.
\begin{proposition}\label{p_AfMb}
Under (T1), (T2) and (T3), $\mathcal{F}_{M}$ is Fr\'echet-differentiable, from the right with respect to $\omega$, at every $(\hat{v},\hat u,\hat\omega)\in\mathbb{V}\times\mathbb{U}\times(0,+\infty)$ and
\begin{equation*}
D\mathcal{F}_{M}(\hat{v},\hat u,\hat\omega)(v,u,\omega)=\mathfrak{L}_{M}(\cdot;\hat{v},\hat\omega)[v_{(\cdot)}\circ s_{\hat\omega}]+\omega \mathfrak{M}_{M}(\cdot;\hat{v},\hat\omega)
\end{equation*}
for $(v,u,\omega)\in\mathbb{V}\times\mathbb{U}\times(0,+\infty)$, where, for $t\in [0,1]$,
\begin{equation*}
\mathfrak{L}_{M}(t;\hat v,\hat\omega)[v_t\circ s_{\hat\omega}]:=\hat\omega\sum_{i=0}^Mw_iD_2K(\hat\omega\eta_i, \hat v(t+\eta_i))v(t+\eta_i)
\end{equation*}
and
\begin{equation*}
\setlength\arraycolsep{0.1em}\begin{array}{rl}
\mathfrak{M}_{M}(t;v,\omega):=&\displaystyle\sum_{i=0}^Mw_i K(\omega\eta_i, v(t+\eta_i))-\displaystyle\frac{\tau}{\omega} K\left(-\tau,v\left(t-\frac{\tau}{\omega}\right)\right)\\[2mm]
+&\displaystyle\omega\sum_{i=0}^Mw_iD_1K(\omega\eta_i,v(t+\eta_i))\eta_i.
\end{array}
\end{equation*}
\end{proposition}
\begin{proof}The proposition can be proved as Proposition \ref{p_Afb}, by replacing $\mathcal{F}$ in the first of \eqref{Phi} with $\mathcal{F}_M$ in \eqref{quadrule}.\end{proof}

\bigskip
For the remaining numerical assumptions, it is useful to define $\Psi_{L,M}:\mathbb{U}\times\mathbb{A}\times\mathbb{B}\rightarrow\mathbb{U}\times\mathbb{A}\times\mathbb{B}$ as
\begin{equation}\label{PsiLM}
\Psi_{L,M}:=I_{\mathbb{U}\times\mathbb{A}\times\mathbb{B}}-P_{L}R_{L}\Phi_{M}.
\end{equation}

\bigskip
The second numerical assumption in \cite{mas15NM} is CS1 (page 536), which is somehow the discrete version of A$x^{\ast}1$ therein, here Proposition \ref{p_Ax*1}. With respect to its validity, the following holds.
\begin{proposition}\label{p_CS1}
Under (T1), (T2), (T3), (T5), (N1) and (N2), there exist $r_{1}\in(0,\omega^{\ast})$ and $\kappa\geq0$ such that
\begin{equation*}
\setlength\arraycolsep{0.1em}\begin{array}{rcl}
\|D\Psi_{L,M}(u,\alpha,\omega)&-&D\Psi_{L,M}(u^{\ast},\alpha^{\ast},\omega^{\ast})\|_{\mathbb{U}\times\mathbb{A}\times\mathbb{B}\leftarrow\mathbb{U}\times\mathbb{A}\times(0,+\infty)}\\[2mm]
&\leq&\kappa\|(u,\alpha,\omega)-(u^{\ast},\alpha^{\ast},\omega^{\ast})\|_{\mathbb{U}\times\mathbb{A}\times\mathbb{B}}
\end{array}
\end{equation*}
for all $(u,\alpha,\omega)\in b((u^{\ast},\alpha^{\ast},\omega^{\ast}),r_{1})$ and for all positive integers $L$ and $M$.
\end{proposition}
\begin{proof} The proof is substantially the same as that of \cite[Proposition 3.7]{ab20mas} since the same primary discretization is used; the main difference lies in the spaces involved, and therefore in the norm in which the left-hand side needs to be evaluated. In practice, we can define $\kappa:=\max\{\Lambda_m,1\}\cdot\max_{i\in\{1,2\},t\in[0,1]}\|D_iK(\omega^*\cdot,v^{\ast}_{t})\|_{\mathbb{R}^{d}\leftarrow \mathbb{R}}$, where $\Lambda_m$ is the Lebesgue constant of the chosen nodes. In particular, unlike the RFDE case, the Lebesgue constant defined by the derivatives of the Lagrange polynomials is not involved.\end{proof}

\bigskip
Correspondigly, the last numerical assumption (CS2, page 537), can be seen as the discrete version of A$x^{\ast}2$ therein, here Proposition \ref{p_Ax*2}. With respect to its validity, the following holds.
\begin{proposition}\label{p_CS2b}
Under (T1), (T2), (T3), (T4), (T5), (N1), (N2) and (N3), the operator $D\Psi_{L,M}(u^{\ast},\alpha^{\ast},\omega^{\ast})$ is invertible and its inverse is uniformly bounded with respect to both $L$ and $M$. Moreover,
\begin{equation*}
\setlength\arraycolsep{0.1em}\begin{array}{rcl}
\lim_{L,M\rightarrow\infty}&&\displaystyle\frac{1}{r_{2}(L,M)}\|[D\Psi_{L,M}(u^{\ast},\alpha^{\ast},\omega^{\ast})]^{-1}\|_{\mathbb{U}\times\mathbb{A}\times\mathbb{B}\leftarrow\mathbb{U}\times\mathbb{A}\times\mathbb{B}}\\[3mm]
&&\cdot\|\Psi_{L,M}(u^{\ast},\alpha^{\ast},\omega^{\ast})\|_{\mathbb{U}\times\mathbb{A}\times\mathbb{B}}=0,
\end{array}
\end{equation*}
where
\begin{equation*}
r_{2}(L,M):=\min\left\{r_{1},\frac{1}{2\kappa\|[D\Psi_{L,M}(u^{\ast},\alpha^{\ast},\omega^{\ast})]^{-1}\|_{\mathbb{U}\times\mathbb{A}\times\mathbb{B}\leftarrow\mathbb{U}\times\mathbb{A}\times\mathbb{B}}}\right\}
\end{equation*}
with $r_{1}$ and $\kappa$ as in Proposition \ref{p_CS1}.
\end{proposition}
\begin{proof}The proof of this proposition is a bit laborious, just like its counterpart in the RFDE case. The latter has been proved in \cite{ab20mas} in several steps, the first of which concerns the invertibility of the operator $D\Psi_{L,M}(u^*,\alpha^*,\omega^*)$ defined in \eqref{PsiLM} for $L,M$ large enough and can be proved as \cite[Proposition 3.11]{ab20mas}. The second step concerns the uniform boundedness of $D\Psi_{L,M}^{-1}(u^*,\alpha^*,\omega^*)$ and follows the ideas of \cite[Lemma 3.12]{ab20mas}. The latter is based on \cite[Proposition A.8]{ab20mas}, which states that
$
\lim_{L,M\to\infty} \omega_{L,M}=\omega,
$
where $\omega_{L,M}$ is the last component of the solution of the discretized version of \eqref{Ax*20}. The limit does not necessarily hold in the present case, however it can be proved that $\vert\omega_{L,M}-\omega\vert$ is uniformly bounded. This follows by the fact that $\|\xi_{L,M,2}^*-\xi_2^*\|_X$ is in turn uniformly bounded (but not necessarily vanishing) thanks to the choice of $\mathbb{U}$ in (T1), where $\xi_{L,M,2}^*$ is the discrete version of $\xi_{2}^*$. As a consequence, the error component called $\varepsilon_{\omega,L,M}$ in \cite[Lemma 3.12]{ab20mas} cannot be proven to vanish in the present case, but would still be uniformly bounded, and that is enough to complete the second step of the proof. The third and last step consists in proving that $\Psi_{L,M}(u^*,\alpha^*,\omega^*)$ vanishes and goes as the proof of \cite[Proposition 3.13]{ab20mas}.\end{proof}
\subsection{Final convergence results}
\label{s_convresult}
From the propositions in the previous subsections we can conclude that our problem of finding a fixed point of $\Phi$ in \eqref{Phi} satisfies all the assumptions required by \cite{mas15NM} under certain hypotheses on the state spaces, the discretization and the regularity of the right-hand side. As a consequence, the relevant FEM converges.
\begin{theorem}[\protect{\cite[Theorem 2, page 539]{mas15NM}}] Under (T1), (T2), (T3), (T4), (T5), (N1), (N2) and (N3), there exists a positive integer $\hat{N}$ such that, for all $L,M\geq\hat{N}$, the operator $R_{L}\Phi_{M}P_{L}$ has a fixed point $(u_{L,M}^{\ast},\alpha_{L,M}^{\ast},\omega_{L,M}^{\ast})$ and
\begin{equation*}
\setlength\arraycolsep{0.1em}\begin{array}{rcl}
\varepsilon_{L,M}:=\|(v_{L,M}^{\ast},\omega_{L,M}^{\ast})-(v^{\ast},\omega^{\ast})\|_{\mathbb{V}\times\mathbb{B}}\leq2&&\cdot\|[D\Psi_{L,M}(u^{\ast},\alpha^{\ast},\omega^{\ast})]^{-1}\|_{\mathbb{U}\times\mathbb{A}\times\mathbb{B}\leftarrow\mathbb{U}\times\mathbb{A}\times\mathbb{B}}\\[2mm]
&&\cdot\|\Psi_{L,M}(u^{\ast},\alpha^{\ast},\omega^{\ast})\|_{\mathbb{U}\times\mathbb{A}\times\mathbb{B}},
\end{array}
\end{equation*}
where $v_{L,M}^*=\mathcal{G}(u_{L,M}^*,\alpha_{L,M}^*)$ and $v^*=\mathcal{G}(u^*,\alpha^*)$.
\end{theorem}
Thanks to Proposition \ref{p_CS2b}, the error on $(v^{\ast},\omega^{\ast})$ is determined by the last factor, namely the \emph{consistency error}. For the latter, thanks to basic results on polynomial interpolation, we can write
\begin{equation*}
\|\Psi_{L,M}(u^{\ast},\alpha^{\ast},\omega^{\ast})\|_{\mathbb{U}\times\mathbb{A}\times\mathbb{B}}\leq\varepsilon_{L}+\max\{\Lambda_{m},1\}\varepsilon_{M},
\end{equation*}
where $\Lambda_m$ is the Lebesgue constant associated to the collocation nodes and the terms
\begin{equation}\label{epsL}
\varepsilon_{L}:=\|(I_{\mathbb{U}\times\mathbb{A}\times\mathbb{B}}-P_{L}R_{L})(u^{\ast},\alpha^{\ast},\omega^{\ast})\|_{\mathbb{U}\times\mathbb{A}\times\mathbb{B}}
\end{equation}
and
\begin{equation}\label{epsM}
\varepsilon_{M}:=\|\Phi_{M}(u^{\ast},\alpha^{\ast},\omega^{\ast})-\Phi(u^{\ast},\alpha^{\ast},\omega^{\ast})\|_{\mathbb{U}\times\mathbb{A}\times\mathbb{B}}.
\end{equation}
are called respectively {\it primary} and {\it secondary} consistency errors.

\bigskip
As for $\varepsilon_{L}$ in \eqref{epsL}, which concerns only the primary discretization, a bound can be obtained from the regularity of $u^{\ast}$ through the following theorem.
\begin{theorem}\label{t_epsL}
Let $K\in\mathcal{C}^{p}(\mathbb{R}\times\mathbb{R}^d,\mathbb{R}^{d})$ for some integer $p\geq0$. Then, Under (T1), (T2), (N1) and (N2), it holds that $u^{\ast}\in C^{p+1}([0,1],\mathbb{R}^{d})$, $\alpha^{\ast}\in C^{p+1}([-1,0],\mathbb{R}^{d})$, $v^{\ast}\in C^{p+1}([-1,1],\mathbb{R}^{d})$ and
\begin{equation}\label{epsLh}
\varepsilon_{L}=O\left(h^{\min\{m,p\}+1}\right).
\end{equation}
\end{theorem}
\begin{proof}
Recall that $v^{\ast}=\mathcal{G}(u^{\ast},\alpha^{\ast})$ satisfies \eqref{rescaledBVPRE}, hence its periodic extension to $[-1,\infty)$ is a periodic solution of \eqref{eq:RE} modulo rescaling of time, and it is  bounded by (T2). Thus, if $K$ is continuous, then the periodic extension of $v^{\ast}$ is continuous in $[0,+\infty)$, thus also in $[-1,\infty)$ by periodicity. Using the continuity of $K$ again, we obtain that $v^*$ is continuously differentiable in $[0,\infty)$, thus also in $[-1,\infty)$ by periodicity. As a consequence, if $p=0$, $v^{\ast}\in C^1([-1,1],\mathbb{R}^d)$. Since $u^*=v^*\vert_{[0,1]}$ and $\alpha^*=v^*\vert_{[-1,1]}$, we immediately have also $u^{\ast}\in C^{1}([0,1],\mathbb{R}^{d})$ and $\alpha^*\in C^1([-1,0],\mathbb{R}^d)$. The whole reasoning can be iterated, proving the first part of the result.

To prove \eqref{epsLh}, we observe first that 
\begin{equation}\label{epsLu1}
\|u^{\ast}-\pi_{L}^{+}\rho_{L}^{+}u^{\ast}\|_{\mathbb{U}}\leq\frac{\|{u^{\ast}}^{(m+1)}\|_{\infty}}{(m+1)!}\cdot h^{m+1}
\end{equation}
holds if $p\geq m$, while
\begin{equation}\label{epsLu2}
\|u^{\ast}-\pi_{L}^{+}\rho_{L}^{+}u^{\ast}\|_{\mathbb{U}}\leq(1+\Lambda_{m})\left(\frac{h}{2}\right)^{p+1}\frac{c_{p+1}}{m^{p+1}}\cdot\|{u^{\ast}}^{(p+1)}\|_{\infty}
\end{equation}
holds if $p\leq m$, with $c_{p+1}$ a positive constant independent of $m$. \eqref{epsLu1} is a direct consequence of the standard Cauchy interpolation reminder, see, e.g., \cite[Section 6.1, Theorem 2]{kc02}. \eqref{epsLu2} is a direct consequence of Jackson's theorem on best uniform approximation, see, e.g., \cite[(2.9) and (2.11)]{mas15I}. The same arguments hold for $\alpha^*$.

\end{proof}
\bigskip
On the other hand, $\varepsilon_{M}$ in \eqref{epsM} concerns only the secondary discretization and is therefore absent whenever the latter is not needed. However, concerning our specific problem, according to \eqref{Phi} and \eqref{PhiM}, it can be written as
\begin{equation}\label{epsMU}
\varepsilon_{M}:=\|F_{M}(v^{\ast}_{\cdot}\circ s_{\omega^{\ast}})-F(v^{\ast}_{\cdot}\circ s_{\omega^{\ast}})\|_{\mathbb{U}}
\end{equation}
and needs to be considered if the integral in \eqref{hpRHS} cannot be exactly computed, in which case \eqref{epsMU} is basically a quadrature error. Assuming that $M$ varies proportionally to $L$, one can choose a formula that guarantees at least the same order of the primary consistency error, so that the order of convergence of the final error is in fact the one given by theorem \ref{t_epsL}.

\begin{remark}\label{r_representation}
In principle, one could discretize the problem by choosing, for each mesh interval, a set of \emph{representation} nodes used to interpolate which are independent from the collocation nodes. That would mean that the unknowns of the discrete problem are given by the values of the relevant functions at the representation nodes, while the equations need to be satisfied at the collocation nodes. If $x^r_{L,M}$ is the vector of the unknowns and $Q_L:X_L\to X$ is the prolongation operator corresponding to the representation nodes (while $P_L,R_L$ refer to the collocation ones), the problem actually reads
$R_LQ_Lx^r_{L,M}=R_L\Phi_MQ_Lx^r_{L,M}.$
Thus, the vector $x^*_{L,M}$ given by the values of the relevant function at the collocation nodes is the solution of the discrete fixed point problem, in fact,
$$
x^*_{L,M}=R_LQ_Lx^r_{L,M}=R_L\Phi_MQ_Lx^r_{L,M}=R_L\Phi_MP_LR_LQ_Lx^r_{L,M}=R_L\Phi_MP_Lx^*_{L,M}.
$$
\end{remark}

\begin{remark}\label{r_altRHS}
The entire convergence analysis can as well be carried out for right-hand sides of the form \eqref{altRHS}. In this case, the different theoretical and numerical assumptions read
\begin{itemize}
\item[(T3)] $k:\mathbb{R}\to\mathbb{R}^d$ is measurable;
\item[(T4)] $f\in\mathcal{C}^1(\mathbb{R}^d,\mathbb{R}^d)$;
\item[(T5)] there exist $r>0$ and $\kappa\geq0$ such that
\begin{equation*}
\setlength\arraycolsep{0.1em}\begin{array}{rcl}
\displaystyle\Bigg\|f'\left(\omega\int_0^{\frac{\tau}{\omega}}k(\theta)v(t-\theta)\dd\theta\right)&-&\displaystyle f'\left(\omega^*\int_0^{\frac{\tau}{\omega^*}}k(\theta)v^*(t-\theta)\dd\theta)\right)
\Bigg\|_{\mathbb{R}^{d}}\\[4mm]
&\leq&\kappa\|(v_t,\omega)-(v^{\ast}_{t}, \omega^{\ast})\|_{X\times\mathbb{R}}
\end{array}
\end{equation*}
for every $(v_t,\omega)\in\overline B((v^{\ast}_{t},\omega^{\ast}),r)$, uniformly with respect to $t\in[0,1]$;
\item[(N4)] $k\in\mathcal{C}(\mathbb{R},\mathbb{R}^{d})$.
\end{itemize}
Moreover, the above can be easily further generalized to the case
\begin{equation}\label{sirsRHS}
F(\alpha)=f\left(\int_{0}^{\tau_1}k_1(\sigma)\alpha(-\sigma)\dd\sigma,\ldots,\int_{0}^{\tau_n}k_n(\sigma)\alpha(-\sigma)\dd\sigma\right).
\end{equation}
\end{remark}
\section{Results}\label{s_results}
This section deals with the numerical computation of periodic solutions of some specific REs from the field of population dynamics. In particular, in Subsection \ref{s_numerical} we will provide experimental proof of the order of convergence \eqref{epsLh}, while in Subsection \ref{s_applications} we will show, by means of two examples, how bifurcations may be detected thanks to the computed periodic solutions.

\subsection{Numerical tests}
\label{s_numerical}

\bigskip
The first RE that we consider is
\begin{equation}\label{eq:re}
x(t)=\frac{\gamma}{2}\int_{-3}^{-1}x(t+\theta)(1-x(t+\theta))\dd\theta,
\end{equation}
for which, as shown in \cite{bdls16}, the exact expression of the periodic solution between a Hopf bifurcation (at $\gamma=2+\pi/2$) and the first period doubling (at $\gamma\approx 4.327$) is $x(t)=\sigma+A\sin\left(\pi t/2\right)$, where $\sigma=1/2+\pi/(4\gamma)$ and $A^2= 2\sigma\left(1-1/\gamma-\sigma\right)$.
The integral representing the distributed delay is approximated through a Clenshaw-Curtis quadrature \cite{tref00} rescaled to the interval $[-3,-1]$.

Starting from the exact solution at $\gamma=4$, the branch of periodic orbits is continued up to the first period doubling after the Hopf bifurcation.
The continuation is performed using a trivial phase condition by forcing $x(0)=\sigma$, and Chebyshev extrema as collocation points. The left panel of Figure \ref{fig:quadratic_errs} confirms the $O(h^{m+1})$ behavior (being $p=+\infty$).
\begin{figure}
\centering
\includegraphics[scale=0.75]{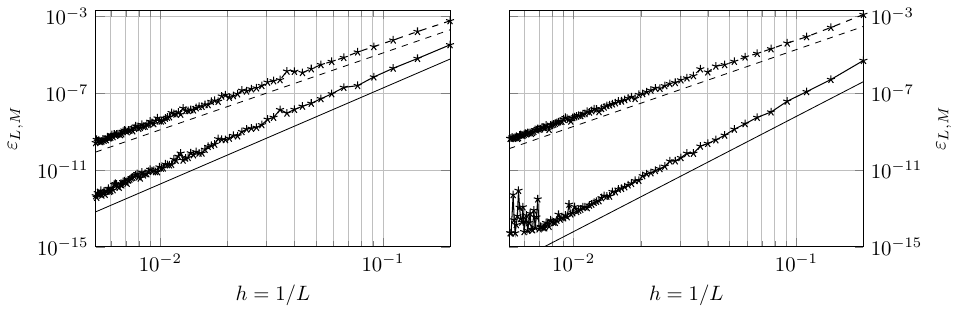}
\caption{Error on the periodic solution of \eqref{eq:re} at $\gamma=4.327$. Left: $m=3$ (dashed line) and $m=4$ (solid line) using Chebyshev points, compared to straight lines having slope 4 (dashed) and 5 (solid). Right: $m=3$ (dashed line) and $m=5$ (solid line) using Gauss-Legendre points. Original figures from \cite{aa20}.}
\label{fig:quadratic_errs}
\end{figure}
Given the experimental proof, found in \cite{elir00}, that in the case of RFDEs the order of convergence increases from $m$ to $m+1$ when using Gauss-Legendre collocation points, we wonder whether a similar increase can be observed in the case of REs. The above experiment is then replicated using such collocation points, and the right panel of Figure \ref{fig:quadratic_errs} confirms the same $O(h^{m+1})$ behavior as the case of Chebyshev extrema.

\bigskip
Next, we show that we obtain the same results with an RE that is defined by a right-hand side of the form \eqref{sirsRHS} as
\begin{equation}\label{sirsRE}
x(t)=\gamma\displaystyle\left(1-\int_0^1x(t-s)\dd s\right)\int_0^1\alpha s^2 e^{-10s}x(t-s)\dd s,
\end{equation}
where $\alpha:=1/\int_0^1s^2 e^{-10s}\dd s=500e^{10}/(e^{10}-61)$, while $\gamma$ is the varying parameter.
As shown in \cite{sdv21}, a Hopf bifurcation occurs when $\log \gamma \approx 1.6553$.

Starting from a perturbation of the equilibrium at the Hopf bifurcation point, the branch of periodic orbits is continued up to $\log\gamma=1.75$. Given the absence of an exact expression of the true solution, unlike the case \eqref{eq:re}, the error is computed with respect to a reference solution which is in turn computed using $L=1000$ and $m=4$. Figure \ref{fig:prova_SIRS} confirms again the $O(h^{m+1})$ behavior regardless of the chosen collocation nodes.
\begin{figure}
\centering
\includegraphics[scale=0.75]{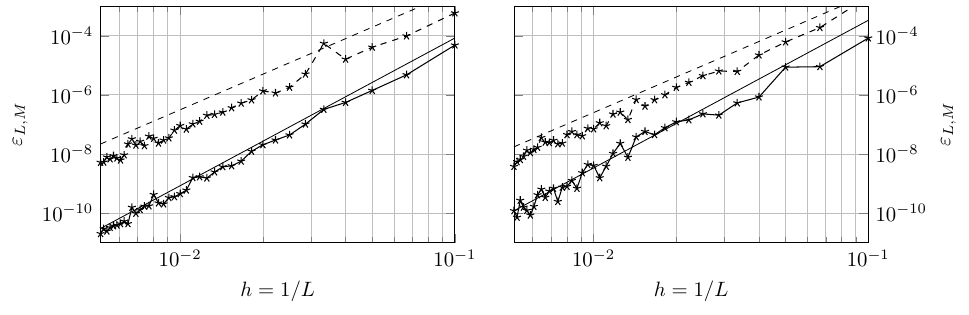}
\caption{Error on the periodic solution of \eqref{sirsRE} at $\log\gamma=1.75$. Left: $m=3$ (dashed line) and $m=4$ (solid line) using Chebyshev points, compared to straight lines having slope 4 (dashed) and 5 (solid). Right: $m=3$ (dashed line) and $m=4$ (solid line) using Gauss-Legendre points, compared to straight lines having slope 4 (dashed) and 5 (solid).}
\label{fig:prova_SIRS}
\end{figure}

\bigskip
Concerning the convergence of the SEM, it is not yet clear whether it can also be proved under the general framework used in the current work (see \cite[Subsection 4.4]{ab20mas} for a brief discussion concerning the RFDE case). However, some numerical experiments run by the authors suggest that the SEM does converge for periodic BVPs defined by REs. In the case of \eqref{sirsRE}, figure \ref{fig:SEM} shows a spectral decay of the error as $m$ increases while $L=1$ remains fixed, although some numerical instability can be observed when using large ($>35$) values of $m$.

\begin{figure}
\centering
\includegraphics[scale=1.15]{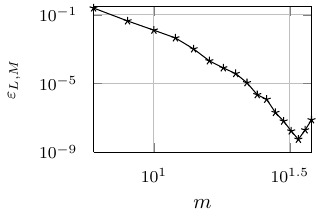}
\caption{Error on the periodic solution of \eqref{sirsRE} at $\log\gamma=1.75$ for $L=1$ using Chebyshev extrema as collocation points.}
\label{fig:SEM}
\end{figure}

\bigskip
Finally, consider the RE given by
\begin{equation}\label{C2notC3}
x(t)=\frac{\gamma}{2}\int_{-3}^{-1}h(x(t+\theta))\dd\theta,
\end{equation}
where
\begin{equation*}
h(x)=\left\{\setlength\arraycolsep{0.1em}\begin{array}{rl}
\displaystyle -x^3+\frac{3}{4}x,\quad &x<\displaystyle\frac{1}{2}\\[2mm]
x^3-3x^2+\displaystyle\frac{9}{4}x-\frac{1}{4},\quad &x\geq\displaystyle\frac{1}{2}.\
\end{array}
\right.
\end{equation*}
Observe that the function $h$ is in $C^2(\mathbb{R},\mathbb{R})$ but not in $C^3(\mathbb{R},\mathbb{R})$, due to a discontinuity of the third derivative at $x=1/2$. Using the method in \cite{sdv21}, it can be shown that a Hopf bifurcation occurs at $\gamma\approx 3.4031$.
\begin{figure}
\centering
\includegraphics[scale=0.75]{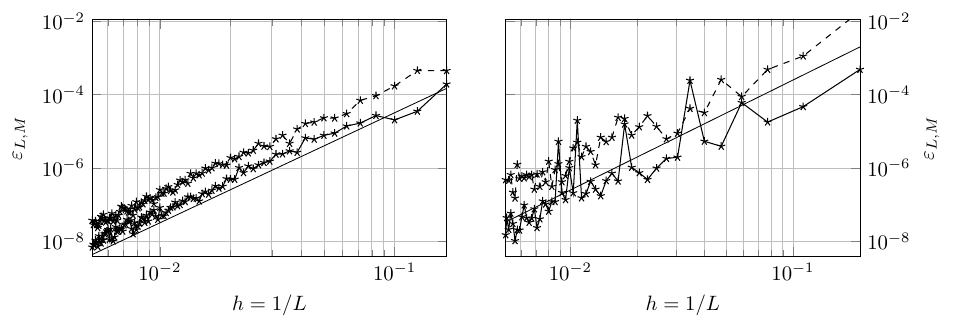}
\caption{Error on the periodic solution of \eqref{C2notC3} at $\gamma=6$. Left: $m=3$ (dashed line) and $m=4$ (solid line) using Chebyshev points, compared to straight line having slope 3. Right: $m=3$ (dashed line) and $m=5$ (solid line) using Gauss-Legendre points.}
\label{fig:prova_C2}
\end{figure}
Starting from a perturbation of the equilibrium at the Hopf bifurcation point, the branch of periodic orbits is continued up to $\gamma=6$. The integrals representing the distributed delays are again approximated through Clenshaw-Curtis quadrature, and the experiment is again performed using both Chebyshev and Gauss-Legendre collocation points. The error is computed with respect to a reference solution which is in turn computed using $L=400$ and $m=4$. Being the values for $m$ considered ($3$ and $4$) greater than $p=2$, the left panel of Figure \ref{fig:prova_C2} confirms the $O(h^{p+1})$ behavior in the former case, as the right one does in the latter.

\subsection{Applications}
\label{s_applications}
As anticipated, the examples in this section show how the method can be used within a continuation framework in view of a bifurcation analysis. The next RE that we consider to this aim is
\begin{equation}\label{expRE}
x(t)=\frac{\gamma}{2}\int_{-3}^{-1}x(t+\theta)e^{-x(t+\theta)}\dd\theta.
\end{equation}
As shown in \cite{sdv21}, a period doubling bifurcation occurs when $\log \gamma \approx 3.8777$. In \cite{bl20} a Floquet theory for REs has been developed and proved valid; in particular, the stability of a periodic solution is determined by the eigenvalues of the monodromy operator of the corresponding linearized equation, according to \cite[Corollary 15]{bl20}. Such eigenvalues can, in turn, be computed numerically thanks to the pseudospectral method in \cite{bl18}, which is used in order to obtain Figure \ref{fig:expRE}. As shown therein, indeed, two stable periodic solutions can be computed on opposite sides of $\log\gamma = 3.8777$, one having roughly double minimal period than the other, thus confirming the presence of a period doubling bifurcation.
\begin{figure}
\centering
\includegraphics[scale=0.6]{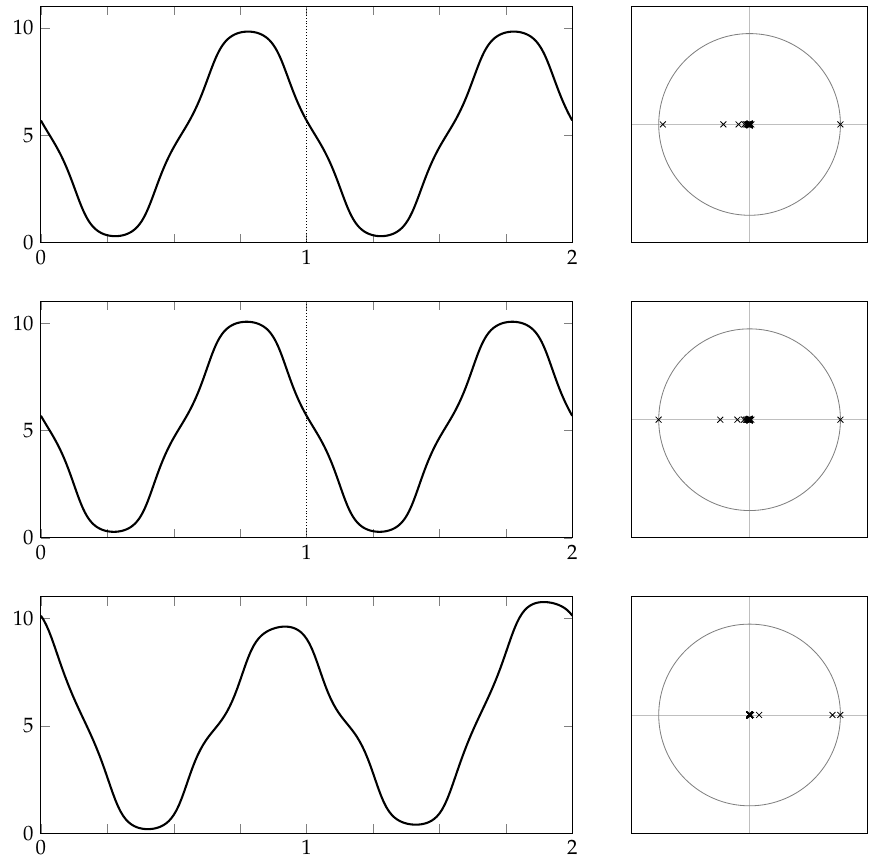}
\caption{Stable periodic solutions of \eqref{expRE} at $\log\gamma=3.83$ (top), $3.8777$ (middle, period doubling) and $3.9$ (bottom), having periods $\omega = 4,\, 4$ and $\approx 8.003$ respectively, computed with $L=20,\,20$ and $40$ and $m=5$. Left: representation of two periods (top, middle) and one period (bottom) of the solutions in the scaled interval $[0,2]$. Right: eigenvalues of the corresponding monodromy operator with respect to the unit circle, all internal in the first and third pictures with the exception of the trivial multiplier 1 (due to linearization, \cite[Proposition 10]{bl20}).}
\label{fig:expRE}
\end{figure}

\bigskip
Period doubling bifurcations also occur in the case \eqref{eq:re}, as shown in \cite{bdls16}. In particular, the second one after the Hopf bifurcation (near which it is not possible to obtain an exact expression of the solution) is detected at $\gamma\approx 4.497$. While, at that point, new stable periodic solutions emerge having double minimal period than the stable old ones, unstable solutions with  (roughly) unchanged period also exist, and can be computed using the method that we are proposing. Figure \ref{fig:quadRE}, indeed, shows that two periodic solutions which are very close to each other can be computed on opposite sides of $\gamma = 4.497$, one being stable and the other being unstable, thus confirming the presence of a bifurcation.
\begin{figure}
\centering
\includegraphics[scale=0.6]{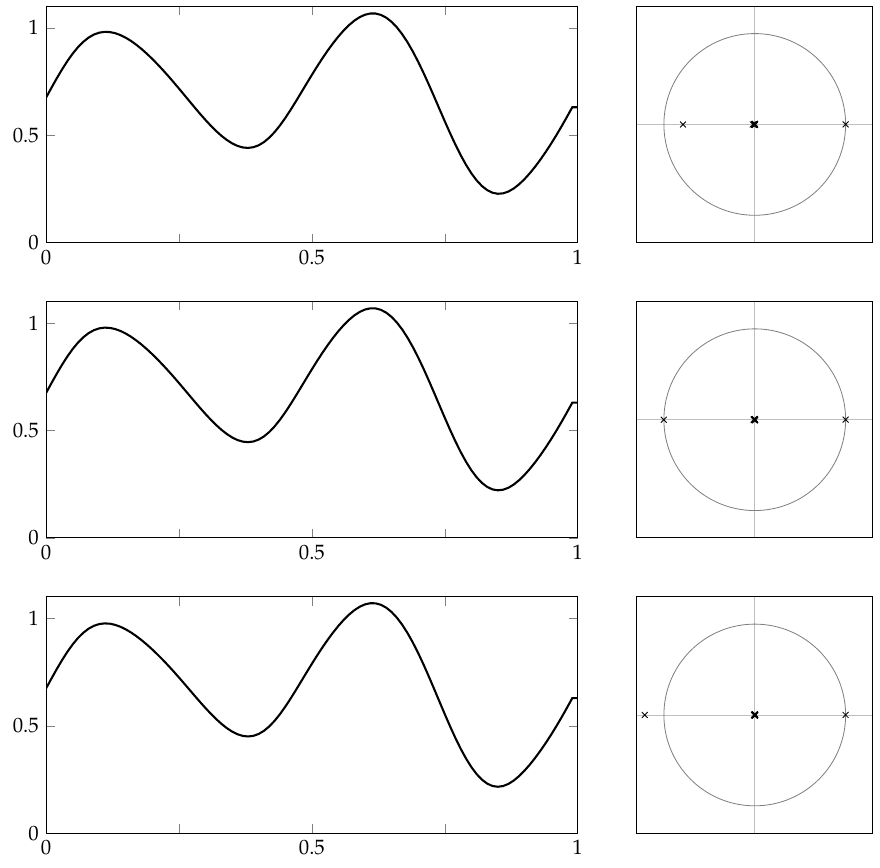}
\caption{Periodic solutions of \eqref{eq:re} at $\gamma=4.48$ (top), $4.497$ (middle, period doubling) and $4.51$ (bottom), having periods $\omega \approx 8.043,\, 8.049$ and $8.056$ repectively, computed with $L=20$ and $m=5$. Left: representation of one period of the solutions in the scaled interval $[0,1]$. Right: eigenvalues of the corresponding monodromy operator with respect to the unit circle, showing the change in stability.}
\label{fig:quadRE}
\end{figure}

\section{Concluding remarks}
\label{s_concluding}
In the past few decades, piecewise orthogonal collocation has been largely used to compute periodic solutions of various classes of delay equations. The present work aims at giving a rigorous description of such a method in the case of REs, while furnishing a complete theoretical error analysis as well as experimental proofs of its validity. In particular, the theoretical proof is based on the abstract approach in \cite{mas15NM} for general BVPs, as it is the case for the work \cite{ab20mas} concerning the corresponding proof for RFDEs. The main result concerns the FEM and states that, for smooth problems, the error decays as $O(L^{-(m+1)})$ where $L$ is the increasing number of mesh intervals, while $m$ is the constant degree of the piecewise polynomials used.

Given both the convergence analysis for RFDEs in \cite{ab20mas} and that in the present paper for REs, we expect to be able to rely on the general approach in \cite{mas15NM} also for the case of coupled RFDE/RE systems, motivated by their predominance in population dynamics. The proof is currently work in progress and the expected main challenge is given by the need to resort to the theory of VIEs with measure kernels \cite[Chapter 10]{grip09}.

Moreover, it would be interesting to extend the method (and, correspondingly, the convergence analysis) to different and more complex classes of delay equations. The first step that we plan to take in this direction is to try to apply the approach to differential equations with non-constant delays (in particular, state-dependent delays), for which the setting defined in \cite{ab20mas} cannot be applied.

\section*{Acknowledgments}
The authors are members of INdAM Research group GNCS and of UMI Research group “Modellistica socio-epidemiologica”. This work was partially supported by the Italian Ministry of University and Research (MUR) through the PRIN 2020 project (No. 2020JLWP23) “Integrated Mathematical Approaches to Socio-Epidemiological Dynamics” (CUP: E15F21005420006).
\appendix

\section{}
\label{s_appendix}
This Appendix completes the proof of Propositon \ref{p_Ax*2} by showing that $\xi_1$ introduced in the proof cannot belong to $R$.  To this aim, we observe that the monodromy operator in Propositon \ref{p_Ax*2} can also be defined in $L^1([-1,0],\mathbb{R}^d)$. The action of the operator remains the same, meaning that $X$ - or better, the corresponding quotient space obtained by considering almost-everywhere equality of functions - is invariant under such action. Thus, in the spectral decomposition $X=K\oplus R$ remains the same. Assume for a contradiction that $\xi_1^*\in R$, i.e., that $\mathfrak{M}_1^{*(0)}=\mathfrak{M}_1^*\in R$, and define $Y:=L^{\infty}([0,1];\mathbb{R}^{d})$\footnote{Here, elements in $\mathbb{R}^{d}$ are intended as row vectors.}. Consider the standard bilinear form $\langle\cdot,\cdot\rangle:Y\times X\to\mathbb{R}$ defined as
\begin{equation}\label{bilinear1}
\langle\psi,\alpha\rangle:=\int_{-1}^{0}\psi(r+\theta)\alpha(\theta)\dd\theta=\int_{0}^{1}\psi(\eta)\alpha(\eta-r)\dd\eta.
\end{equation}
As a general fact, any left eigenvector of some operator $A$ w.r.t. some eigenvalue $\lambda$ is orthogonal - in the corresponding bilinear form - to any element in the range of the operator $I-\lambda A$. This means that, by pairing $\mathfrak{M}_1^*$ with any left eigenvector $\psi$ of $T^{\ast}(1,0)$ with respect to \eqref{bilinear1}, we get
\begin{equation}\label{contra}
\int_{0}^{1}\psi(\eta)\mathfrak{M}^{\ast}(\eta)\dd\eta=\int_{0}^{1}\psi(\eta)\mathfrak{M}_{1}^{\ast}(\eta-1)\dd\eta=\langle\psi,\mathfrak{M}_{1}^{\ast}\rangle=0.
\end{equation}
In order to obtain a contradiction with \eqref{contra}, we resort to adjoint theory for VIEs (see \cite{grip09} as a general reference). Indeed, any RE of the form
\begin{equation*}
x(t)=\int_{t-r}^{t}K^{\ast}(t,\sigma-t)x(\sigma)\dd\sigma,\quad t\geq t_{0},
\end{equation*}
with $x_{t_{0}}=\alpha$ for some $\alpha\in X:=L^{1}([-1,0];\mathbb{R}^{d})$\footnote{In our case $K^{\ast}(t,\sigma-t):=\omega^{\ast}D_{2}K(\omega^{\ast}(\sigma-t),v^{\ast}(\sigma))$ for $t\geq t_{0}$ and $\sigma\in[t-r,t]$ with $r:=\tau/\omega^{\ast}\leq1$.} and $r\leq 1$ can be written as the VIE
\begin{equation}\label{VIE}
x(t)=\int_{t_{0}}^{t}K_{0}^{\ast}(t,\sigma)x(\sigma)\dd\sigma+f(t),\quad t\geq t_{0},
\end{equation}
where
\begin{equation*}
K_{0}^{\ast}(t,\sigma):=\begin{cases}
K^{\ast}(t,\sigma-t),&t\geq t_{0}\textrm{ and }\sigma\in[t-r,t],\\
0,&\textrm{otherwise},
\end{cases}
\end{equation*}
\begin{equation*}
\alpha_{0}(t_{0}+\theta):=\begin{cases}
\alpha(\theta),&\theta\in[-1,0],\\
0,&\textrm{otherwise},
\end{cases}
\end{equation*}
and
\begin{equation*}
f(t):=\int_{t-1}^{t_{0}}K_{0}^{\ast}(t,\sigma)\alpha_{0}(\sigma)\dd\sigma,\quad t\geq t_{0}.
\end{equation*}
Existence and uniqueness \cite[Chapter 9]{grip09} allow to define the forward evolution family $\{T(t,t_{0})\}_{t\geq t_{0}}$ on $X$ through $T(t,t_{0})\alpha=x_{t}$. From \cite[Exercise 6, p.274]{grip09}, we have the adjoint VIE\footnote{It is enough to consider the integrals at the right-hand side of the VIE and of its adjoint over the whole of $\mathbb{R}$, by taking into account the definition of $K_{0}^{\ast}$.}
\begin{equation}\label{adjeq}
y(s)=\int_{s}^{s_{0}}y(\sigma)K_{0}^{\ast}(\sigma,s)\dd\sigma+g(s),\quad s\leq s_{0},
\end{equation}
with
\begin{equation*}
g(s):=\int_{s_{0}}^{s+1}\psi_{0}(\sigma)K_{0}^{\ast}(\sigma,s)\dd\sigma,\quad s\leq s_{0},
\end{equation*}
for
\begin{equation*}
\psi_{0}(s_{0}+\eta):=\begin{cases}
\psi(\eta),&\eta\in[0,1],\\
0,&\textrm{otherwise},
\end{cases}
\end{equation*}
and $y^{s_{0}}=\psi$ for some $\psi\in Y$, where we use the notation $y^s(\eta):=y(s+\eta)$ for $\eta\in[0,r]$. Then, one defines the backward evolution family $\{V(s,s_{0})\}_{s\leq s_{0}}$ on $Y$ through $V(s,s_{0})\psi=y^{s}$.

\bigskip
From the theory of resolvents \cite[Chapter 9, Section 3]{grip09}, we can express the solution of \eqref{VIE} and that of \eqref{adjeq} respectively as
\begin{equation}\label{xf}
x(t)=f(t)+\int_{t_{0}}^{t}R_{0}^{\ast}(t,\sigma)f(\sigma)\dd\sigma,\quad t\geq t_{0},
\end{equation}
and
\begin{equation}\label{yg}
y(s)=g(s)+\int_{s}^{s_{0}}g(\sigma)R_{0}^{\ast}(\sigma,s)\dd\sigma,\quad s\leq s_{0},
\end{equation}
where $R_{0}^{\ast}$ is the resolvent of \eqref{VIE}.

\bigskip
Given $t\in\mathbb{R}$, consider now the pairing $[\cdot,\cdot]_{t}:Y\times X\to\mathbb{R}$ defined as
\begin{equation}\label{bilinear2}
[\psi,\alpha]_{t}:=\int_{0}^{1}\psi(\eta)\int_{-1}^{0}K_0^{\ast}(t+\eta,t+\beta)\alpha(\beta)\dd\beta\dd\eta.
\end{equation}

\noindent Observe that such bilinear form is nondegenerate for all $t\in\mathbb{R}$ whenever $K^*$ (and thus $K_0^*$) is nontrivial. Indeed, assume by contradiction that there exists $\psi\in Y$ such that $\psi$ is nonzero, but $[\psi,\cdot]_t$ is constantly 0. By the nondegenerateness of the bilinear form $\langle\cdot,\cdot\rangle$, this means that the innermost integral is 0 for all $\alpha\in X$ and almost all $\eta\in[0,1]$. If $\alpha:=x_t$, where $x$ is the (unique modulo multiplication by constant) 1-periodic solution of the VIE, then such integral is equal to
\begin{equation*}
\int_{-1}^0K_0^*(t+\eta,t+\beta)x(t+\beta)\dd\beta=\int_{t-1}^tK_0^*(t+\eta,\beta)x(\beta)\dd\beta=x(t+\eta).
\end{equation*}
Thus $x_{t+1}$ is almost everywhere equal to 0. Using periodicity, this means that $x$ is almost everywhere 0, which is only possible if $K^*$ is trivial, contradiction. Using similar arguments, one can prove that there is no nonzero $\alpha\in X$ such that $[\cdot,\alpha]_t$ is constantly zero, after exchanging the order of integration in the definition of $[\cdot,\cdot]_t$.

\bigskip
We claim that the forward monodromy operator and the corresponding backward one are adjoint w.r.t. \eqref{bilinear2}, i.e., that
\begin{equation}\label{VT2}
[V(t-1,t)\psi,\alpha]_{t}=[\psi,T(t+1,t)\alpha]_{t}.
\end{equation}
Indeed, using \eqref{yg}, we have
\begin{equation*}
\setlength\arraycolsep{0.1em}\begin{array}{rcl}
[V(t-1,t)\psi,\alpha]_{t}&=&\displaystyle\int_{0}^{1}[V(t-1,t)\psi](\eta)\int_{-1}^{0}K_0^{\ast}(t+\eta,t+\beta)\alpha(\beta)\dd\beta\dd\eta\\[4mm]
&=&\displaystyle\int_{0}^{1}y(t-1+\eta;t,\psi)\int_{-1}^{0}K_0^{\ast}(t+\eta,t+\beta)\alpha(\beta)\dd\beta\dd\eta\\[4mm]
&=&\displaystyle\int_{0}^{1}g(t-1+\eta)\int_{-1}^{0}K_0^{\ast}(t+\eta,t+\beta)\alpha(\beta)\dd\beta\dd\eta\\[4mm]
&&+\displaystyle\int_{0}^{1}\int_{t-1+\eta}^{t}g(\sigma)R_{0}^{\ast}(\sigma,t-1+\eta)\dd\sigma\\[4mm]
&&\displaystyle\int_{-1}^{0}K_0^{\ast}(t+\eta,t+\beta)\alpha(\beta)\dd\beta\dd\eta
\\[4mm]
&=&A+B
\end{array}
\end{equation*}
for
\begin{equation*}
A:=\int_{0}^{1}g(t-1+\eta)\int_{-1}^{0}K_0^{\ast}(t+\eta,t+\beta)\alpha(\beta)\dd\beta\dd\eta
\end{equation*}
and 
\begin{equation*}
B:=\int_{0}^{1}\int_{t-1+\eta}^{t}g(\sigma)R_{0}^{\ast}(\sigma,t-1+\eta)\dd\sigma\int_{-r}^{0}K_0^{\ast}(t+\eta,t+\beta)\alpha(\beta)\dd\beta\dd\eta.
\end{equation*}
As for $A$, we have
\begin{equation*}
\setlength\arraycolsep{0.1em}\begin{array}{rcl}
A&=&\displaystyle\int_{0}^{1}\int_{t}^{t+\eta}\psi_{0}(\sigma)K_{0}^{\ast}(\sigma,t-1+\eta)\dd\sigma\int_{-1}^{0}K_0^{\ast}(t+\eta,t+\beta)\alpha(\beta)\dd\beta\dd\eta\\[4mm]
&=&\displaystyle\int_{0}^{1}\int_{0}^{\eta}\psi_{0}(t+\sigma)K_{0}^{\ast}(t+\sigma,t-1+\eta)\dd\sigma\int_{-1}^{0}K_0^{\ast}(t+\eta,t+\beta)\alpha(\beta)\dd\beta\dd\eta\\[4mm]

&=&\displaystyle\int_{0}^{1}\psi_{0}(t+\sigma)\int_{\sigma}^{1}K_{0}^{\ast}(t+\sigma,t-1+\eta)\int_{-1}^{0}K_0^{\ast}(t+\eta,t+\beta)\alpha(\beta)\dd\beta\dd\eta\dd\sigma\\[4mm]
&=&\displaystyle\int_{0}^{1}\psi(\sigma)\int_{\sigma}^{1}K_{0}^{\ast}(t+\sigma,t-1+\eta)\int_{-1}^{0}K_0^{\ast}(t+\eta,t+\beta)\alpha(\beta)\dd\beta\dd\eta\dd\sigma\\[4mm]
&=&\displaystyle\int_{0}^{1}\psi(\sigma)\int_{-1}^{0}\int_{\sigma}^{1}K_{0}^{\ast}(t+\sigma,t-1+\eta)K_0^{\ast}(t+\eta,t+\beta)\dd\eta\alpha(\beta)\dd\beta\dd\sigma\\[4mm]
&=&\displaystyle\int_{0}^{1}\psi(\sigma)\int_{-1}^{0}\int_{0}^{1}K_{0}^{\ast}(t+\sigma,t-1+\eta)K_0^{\ast}(t+\eta,t+\beta)\dd\eta\alpha(\beta)\dd\beta\dd\sigma\\[4mm]
&=&\displaystyle\int_{0}^{1}\psi(\sigma)\int_{-1}^{0}\int_{-1}^{0}K_0^{\ast}(t+\sigma,t+\eta)K_{0}^{\ast}(t+1+\eta,t+\beta)\dd\eta\ \alpha(\beta)\dd\beta\dd\sigma\\[4mm]
&=&\displaystyle\int_{0}^{1}\psi(\eta)\int_{-1}^{0}\int_{-1}^{0}K_0^{\ast}(t+\eta,t+\beta)K_{0}^{\ast}(t+1+\beta,t+\sigma)\dd\beta\ \alpha(\sigma)\dd\sigma\dd\eta\\[4mm]
&=&\displaystyle\int_{0}^{1}\psi(\eta)\int_{-1}^{0}\int_{-1}^{\sigma}K_0^{\ast}(t+\eta,t+\beta)K_{0}^{\ast}(t+1+\beta,t+\sigma)\dd\beta\ \alpha(\sigma)\dd\sigma\dd\eta\\[4mm]

&=&\displaystyle\int_{0}^{1}\psi(\eta)\int_{-1}^{0}K_0^{\ast}(t+\eta,t+\beta)\int_{\beta}^{0}K_{0}^{\ast}(t+1+\beta,t+\sigma)\alpha_{0}(t+\sigma)\dd\sigma\dd\beta\dd\eta\\[4mm]

&=&\displaystyle\int_{0}^{1}\psi(\eta)\int_{-1}^{0}K_0^{\ast}(t+\eta,t+\beta)\int_{t+\beta}^{t}K_{0}^{\ast}(t+1+\beta,\sigma)\alpha_{0}(\sigma)\dd\sigma\dd\beta\dd\eta\\[4mm]

&=&\displaystyle\int_{0}^{1}\psi(\eta)\int_{-1}^{0}K_0^{\ast}(t+\eta,t+\beta)f(t+1+\beta)\dd\beta\dd\eta,
\end{array}
\end{equation*}
where the first equality comes from the definition of $g$, the second follows from the substitution $\sigma\leftarrow t+\sigma$, the third is obtained by exchanging the order of integration between $\eta$ and $\sigma$, the fourth follows from the definition of $\psi_{0}$, the fifth is obtained  by exchanging the order of integration between $\eta$ and $\beta$, the sixth is due to the fact that $K_{0}^{\ast}(t+\sigma,t-1+\eta)$ vanishes for $\eta<\sigma$, the seventh follows from the substitution $\eta\leftarrow 1+\eta$, the eighth is obtained by just renaming the variables, the ninth is due to the fact that $K_0^*(t+1+\beta,t+\sigma)$ vanishes for $\beta> \sigma$, the tenth is obtained by exchanging the order of integration between $\beta$ and $\sigma$, the eleventh follows from the substitution $\sigma\leftarrow \sigma-t$ and the last comes from the definition of $f$. As for $B$, we have
\begin{equation*}
\setlength\arraycolsep{0.1em}\begin{array}{rcl}
B&=&\displaystyle\int_{0}^{1}\int_{t-1+\eta}^{t}\int_0^{\sigma+1-t}\psi(\theta)K_0^*(t+\theta,\sigma)\dd\theta R_{0}^{\ast}(\sigma,t-1+\eta)\dd\sigma\ h(t,\eta)\dd\eta\\[4mm]
&=&\displaystyle\int_{0}^{1}\int_{0}^{\eta}\psi(\theta)\int_{t-1+\eta}^{t}K_0^*(t+\theta,\sigma) R_{0}^{\ast}(\sigma,t-1+\eta)\dd\sigma\dd\theta\ h(t,\eta)\dd\eta\\[4mm]
&&+\displaystyle\int_{0}^{1}\int_{\eta}^{1}\psi(\theta)\int_{t-1+\theta}^{t}K_0^*(t+\theta,\sigma) R_{0}^{\ast}(\sigma,t-1+\eta)\dd\sigma\dd\theta\ h(t,\eta)\dd\eta\\[4mm]
&=&\displaystyle\int_{0}^{1}\int_{0}^{\eta}\psi(\theta)\int_{t-1}^{t}K_0^*(t+\theta,\sigma) R_{0}^{\ast}(\sigma,t-1+\eta)\dd\sigma \dd\theta h(t,\eta)\dd\eta\\[4mm]
&&+\displaystyle\int_{0}^{1}\int_{\eta}^{r}\psi(\theta)\int_{t-1}^{t}K_0^*(t+\theta,\sigma) R_{0}^{\ast}(\sigma,t-1+\eta)\dd\sigma\dd\theta\ h(t,\eta)\dd\eta.\\[4mm]
&=&\displaystyle\int_{0}^{1}\int_{0}^{1}\psi(\theta)\int_{t-1}^{t}K_0^*(t+\theta,\sigma) R_{0}^{\ast}(\sigma,t-1+\eta)\dd\sigma\dd\theta\ h(t,\eta)\dd\eta\\[4mm]
&=&\displaystyle\int_{0}^{1}\int_{0}^{1}\psi(\eta)\int_{t-1}^{t}K_0^*(t+\eta,\beta) R_{0}^{\ast}(\beta,t-1+\sigma)\dd\beta\dd\eta\ h(t,\sigma)\dd\sigma\\[4mm]
&=&\displaystyle\int_{0}^{1}\psi(\eta)\int_{t-1}^{t}K_0^{\ast}(t+\eta,\beta)\int_{0}^{1}R_{0}^{\ast}(\beta,t+\sigma-1)h(t,\sigma)\dd\sigma\dd\beta\dd\eta\\[4mm]
&=&\displaystyle\int_{0}^{1}\psi(\eta)\int_{t-1}^{t}K_0^{\ast}(t+\eta,\beta)\int_{0}^{1}R_{0}^{\ast}(1+\beta,t+\sigma)\\[4mm]
&&\displaystyle\int_{-1}^0K_0^*(t+\sigma,t+\theta)\alpha(\theta)\dd\theta\dd\sigma\dd\beta\dd\eta\\[4mm]
&=&\displaystyle\int_{0}^{1}\psi(\eta)\int_{t-1}^{t}K_0^{\ast}(t+\eta,\beta)\int_{t}^{t+1}R_{0}^{\ast}(1+\beta,\sigma)\\[4mm]
&&\displaystyle\int_{-1}^0K_0^*(\sigma,t+\theta)\alpha(\theta)\dd\theta\dd\sigma\dd\beta\dd\eta
\\[4mm]
&=&\displaystyle\int_{0}^{1}\psi(\eta)\int_{t-1}^{t}K_0^{\ast}(t+\eta,\beta)\int_{t}^{1+\beta}R_{0}^{\ast}(1+\beta,\sigma)\\[4mm]
&&\displaystyle\int_{\sigma-t-1}^0K_0^*(\sigma,t+\theta)\alpha(\theta)\dd\theta\dd\sigma\dd\beta\dd\eta\\[4mm]
&=&\displaystyle\int_{0}^{1}\psi(\eta)\int_{-1}^{0}K_0^{\ast}(t+\eta,t+\beta)\int_{t}^{t+1+\beta}R_{0}^{\ast}(t+1+\beta,\sigma)\\[4mm]
&&\displaystyle\int_{\sigma-1}^tK_0^*(\sigma,\theta)\alpha_0(\theta)\dd\theta\dd\sigma\dd\beta\dd\eta\\[4mm]
&=&\displaystyle\int_{0}^{1}\psi(\eta)\int_{-1}^{0}K_0^{\ast}(t+\eta,t+\beta)\int_{t}^{t+1+\beta}R_{0}^{\ast}(t+1+\beta,\sigma)f(\sigma)\dd\sigma\dd\beta\dd\eta,
\end{array}
\end{equation*}
where the first equality comes from the definition of $g$, the second is obtained by exchanging the order of integration between $\theta$ and $\sigma$, the third is due to the fact that $R_{0}^{\ast}(\sigma,t-1+\eta)$ vanishes if $\sigma<t-1+\eta$ and $K_0^*(t+\theta,\sigma)$ vanishes if $\sigma<t-1+\theta$, the sixth is obtained by changing the order of integration, the seventh follows from the $1$-periodicity of $R_0^*$, the eighth follows from the substitution $\sigma\leftarrow \sigma-t$, the ninth is due to the fact that $R_{0}^{\ast}(1+\beta,\sigma)$ vanishes if $\sigma>1+\beta$ and $K_0^*(\sigma,t+\theta)$ vanishes if $\theta<\sigma-t-1$, the tenth follows from the substitutions $\theta\leftarrow\theta-t$ and $\beta\leftarrow \beta-t$ and the last comes from the definition of $f$. Eventually, using \eqref{xf}, we have
\begin{equation*}
\setlength\arraycolsep{0.1em}\begin{array}{rcl}
A+B&=&\displaystyle\int_{0}^{1}\psi(\eta)\int_{-1}^{0}K_0^{\ast}(t+\eta,t+\beta)f(t+1+\beta)\dd\beta\dd\eta\\[4mm]
&&+\displaystyle\int_{0}^{1}\psi(\eta)\int_{-1}^{0}K_0^{\ast}(t+\eta,t+\beta)\\[4mm]
&&\displaystyle\int_{t}^{t+1+\beta}R_{0}^{\ast}(t+1+\beta,\sigma)f(\sigma)\dd\sigma\dd\beta\dd\eta\\[4mm]
&=&\displaystyle\int_{0}^{1}\psi(\eta)\int_{-1}^{0}K_0^{\ast}(t+\eta,t+\beta)x(t+1+\beta;t,\alpha)\dd\beta\dd\eta\\[4mm]
&=&\displaystyle\int_{0}^{1}\psi(\eta)\int_{-1}^{0}K_0^{\ast}(t+\eta,t+\beta)[T(t+1,t)\alpha](\beta)\dd\beta\dd\eta\\[4mm]
&=&[\psi,T(t+1,t)\alpha]_{t},
\end{array}
\end{equation*}
which proves \eqref{VT2}.

\bigskip
Under the assumption that $1$ is a simple eigenvalue, both the VIE and its adjoint have a unique $1$-periodic solution modulo multiplication by some constant, say $x$ and $y$ respectively. Thus, the associated states $y^t$ and $x_t$ are respectively the left and right $1$-eigenvectors of the operator $T(t+1,t)$. Again thanks to their uniqueness, we have $[y^t,x_t]_t\neq 0$ for all $t\in\mathbb{R}$. Moreover, the continuity of the map $t\mapsto[y^t,x_t]_t$ and the mean value theorem for definite integrals let us conclude that $\int_0^1[y^t,x_t]_t\dd t\neq 0$. Finally, observe that
\begin{equation*}
\setlength\arraycolsep{0.1em}\begin{array}{rcll}
\mathfrak{M}_{}^{\ast}(t)&=&\displaystyle \frac{\dd}{\dd\omega}\left[\int_{-\tau}^{0}K\left(\sigma,v\left(t+\frac{\sigma}{\omega}\right)\right)\dd \sigma\right]\Bigg\vert_{(\omega,v)=(\omega^*,v^*)}\\[3mm]
&=&\displaystyle\int_{-\tau}^{0}D_2K\left(\sigma,v^*\left(t+\frac{\sigma}{\omega^*}\right)\right)v^*\left(t+\frac{\sigma}{\omega^*}\right)\left(-\frac{\sigma}{\omega^{*2}}\right)\dd \sigma\\[3mm]
&=&\displaystyle -\frac{1}{\omega^*}\int_{-\tau}^{0}D_2K\left(\sigma,v^*\left(t+\frac{\sigma}{\omega^*}\right)\right)(v^*)'\left(t+\frac{\sigma}{\omega^*}\right)\left(\frac{\sigma}{\omega^{*}}\right)\dd \sigma\\[3mm]
&=&\displaystyle -\frac{1}{\omega^*}\cdot\omega^*\int_{-\frac{\tau}{\omega^*}}^{0}D_2K(\omega^*\theta,v^*(t+\theta))(v^*)'(t+\theta)\theta\dd \theta\\[3mm]
&=&\displaystyle -\int_{-\frac{\tau}{\omega^*}}^{0}D_2K(\omega^*\theta,v^*(t+\theta))(v^*)'(t+\theta)\theta\dd \theta\\[3mm]
&=&\displaystyle -\frac{1}{\omega^*}\int_{-1}^{0}K_0^*(t,t+\theta)(v^*)'(t+\theta)\theta\dd \theta,
\end{array}
\end{equation*}
where $(v^*)'$ is indeed $x$. Thus
\begin{equation*}
\setlength\arraycolsep{0.1em}\begin{array}{rcl}
0\neq\displaystyle\int_{0}^{1}[y^{t},x_{t}]_{t}\dd t&=&\displaystyle\int_{0}^{1}\int_{0}^{1}y(t+\eta)\int_{-1}^{0}K_0^{\ast}(t+\eta,t+\beta)x(t+\beta)\dd\beta\dd\eta\dd t\\[4mm]
&=&\displaystyle\int_{0}^{1}y(t)\int_{0}^{1}\int_{-1}^{0}K_0^{\ast}(t,t+\beta-\eta)x(t+\beta-\eta)\dd\beta\dd\eta\dd t\\[4mm]
&=&\displaystyle\int_{0}^{1}y(t)\int_{0}^{1}\int_{-\eta-1}^{-\eta}K_0^{\ast}(t,t+\theta)x(t+\theta)\dd\theta\dd\eta\dd t\\[4mm]
&=&\displaystyle\int_{0}^{1}y(t)\int_{0}^{1}\int_{-1}^{-\eta}K_0^{\ast}(t,t+\theta)x(t+\theta)\dd\theta\dd\eta\dd t\\[4mm]
&=&\displaystyle\int_{0}^{1}y(t)\int_{-1}^{0}\int_{0}^{-\theta}K_0^{\ast}(t,t+\theta)x(t+\theta)\dd\eta\dd\theta\dd t\\[4mm]
&=&\displaystyle\int_{0}^{1}y(t)\int_{-1}^{0}K_0^{\ast}(t,t+\theta)x(t+\theta)\int_{0}^{-\theta}\dd\eta\dd\theta\dd t\\[4mm]
&=&-\displaystyle\int_{0}^{1}y(t)\int_{-1}^{0}K_0^{\ast}(t,t+\theta)x(t+\theta)\theta\dd\theta\dd t\\[4mm]
&=&\omega^*\displaystyle\int_{0}^{1}y(t)\mathfrak{M}^{\ast}(t)\dd t,
\end{array}
\end{equation*}
which contradicts \eqref{contra} thanks to the fact that $y^0$ is the left $1$-eigenvector of $T^*(1,0)$.



\end{document}